\documentclass[reqno]{amsart}
\usepackage{amsmath,amsthm,amssymb,amsfonts,mathrsfs,amscd}
\usepackage{multirow}
\usepackage{graphicx}
\usepackage{bm}
\usepackage{mathrsfs}
\usepackage{dsfont}
\usepackage{tikz}
\usepackage{tikz-3dplot}
\usepackage{subfig}
\usepackage{stmaryrd}

\usepackage{enumerate}

\usepackage{diagbox}
\usepackage{algorithm}
\usepackage{algorithmicx}
\usepackage{algpseudocode}
\usepackage{booktabs}
\allowdisplaybreaks[4]

\newtheorem{remark}{Remark}[section]
\newtheorem{assumption}{Assumption}[section]
\newtheorem{lemma}{Lemma}[section]
\newtheorem{theorem}{Theorem}[section]

\makeatletter
  \@addtoreset{equation}{section}
  \newcommand\figcaption{\def\@captype{figure}\caption}
  \newcommand\tabcaption{\def\@captype{table}\caption}
\makeatletter

\begin{document}

\title{A plane wave method based on approximate wave directions for two dimensional Helmholtz equations with large wave numbers}

\author{Qiya Hu}
%\thanks{hqy@lsec.cc.ac.cn}
%\and 
\author{Zezhong Wang}
%\thanks{bater1@yeah.net}
%}
\thanks{1. LSEC, ICMSEC, Academy of Mathematics and Systems Science, Chinese Academy of Sciences, Beijing
100190, China; 2. School of Mathematical Sciences, University of Chinese Academy of Sciences, Beijing 100049,
China (hqy@lsec.cc.ac.cn, bater1@yeah.net). This work was funded by Natural Science Foundation of China G12071469.}

\maketitle
{\bf Abstract.} In this paper we present and analyse a high accuracy method for computing wave directions defined in the geometrical optics ansatz of Helmholtz equation with variable wave number.
Then we define an ``adaptive" plane wave space with small dimensions, in which each plane wave basis function is determined by such an approximate wave direction.
We establish a best $L^2$ approximation of the plane wave space for the analytic solutions of homogeneous Helmholtz equations with large wave numbers and report
some numerical results to illustrate the efficiency of the proposed method.

%We apply the proposed plane wave spaces to the discretization of nonhomogeneous Helmholtz equations with variable wave numbers.



\pagestyle{myheadings}
\thispagestyle{plain}
\markboth{}{}

\section{Introduction}
In this paper we consider the following Helmholtz equation with impedance boundary condition
\begin{equation}\label{nonhomogeneousHelm}
\left\{
	\begin{aligned}
	&\mathcal{L} u  = -(\Delta +\kappa^2(\mathbf{r})) u(\omega,\mathbf{r}) = f(\omega,\mathbf{r}), \quad \mathbf{r} = (x, y) \in\Omega, \\
	&(\partial_{\mathbf{n}} + i\kappa(\mathbf{r})) u(\omega,\mathbf{r}) = g(\omega, \mathbf{r}), \quad \mathbf{r}\in\partial\Omega,
	\end{aligned}
	\right.
\end{equation}
where $\Omega\subset \mathbb{R}^2$ is a bounded Lipchitz domain, $\mathbf{n}$ is the out normal vector on $\partial\Omega$, $f\in L^2(\Omega)$ is the source term and
$\kappa(\mathbf{r})={\omega\over c(\mathbf{r})}$,
$g\in L^2(\partial\Omega)$.
In applications, $\omega$ denotes the frequency and may be large, $c(\mathbf{r})>0$ denotes the light speed, which is usually a variable positive function.
The number $\kappa(\mathbf{r})$ is called the {\it wave number}.

Helmholtz equation is the basic model in sound propagation. It is a very important
topic to design an efficient finite element method for Helmholtz equations with large wave numbers
such that the so called {\it pollution effect} can be reduced. There are some finite element methods that can reduce the pollution effect, for example,
the $hp$ finite element methods \cite{Chen2013, DuWu2015, Wu2011, HuS2020, fem} and the plane wave methods \cite{Buffa2008,Cessenat1998,Farhat2003, Gabard2007}, \cite{Gittelson2009}-\cite{Huqy2018}, \cite{Huttunen2007, Monk1999, Perugia2016}.
 It seems that the plane wave methods have less pollution effect
for homogeneous Helmholtz equations (and time-harmonic Maxwell equations) with constant (or piecewise constant) large wave numbers
%(see \cite{Hiptmair2016})
since plane wave basis functions are solutions of a homogeneous Helmholtz equation without boundary condition and can capture at the maximum the oscillating characteristic of the analytic solution of the original Helmholtz equation. In \cite{Huqy2018}, a plane wave method combined
with local spectral elements was proposed for the discretization of the nonhomogeneous Helmholtz equation (and
time-harmonic Maxwell equations) with (piecewise) constant wave numbers. The basic ideas in this method can be described as follows. At first nonhomogeneous Helmholtz equations on small subdomains
are discretized in the space consisting of higher order polynomials, then the resulting residue Helmholtz equation (which is homogeneous on each
element) on the global solution domain is discretized by the plane wave method. By using this method, we need only to study plane wave method for Helmholtz equations that are homogeneous on every element. Then
we can simply consider the homogeneous Helmholtz equation (on each element)
\begin{equation}\label{homohelm}
(\Delta +\kappa^2(\mathbf{r})) u(\omega,\mathbf{r}) =0, \quad \mathbf{r} = (x, y) \in\Omega.
\end{equation}
Recently, the Geometric Optics Ansatz-based plane wave discontinuous Galerkin (GOPWDG) method for (\ref{homohelm}) was proposed in \cite{HuW2021}, in which a high accuracy $L^2$ error of approximate solution was
established for the case with a variable wave number $\kappa({\bf r})$.

It is known that $L^2$ errors of the finite element solutions of a usual elliptic
equation converge to zero when $h\rightarrow 0$. However, this property does not hold for all the methods mentioned above: the $L^2$ errors of the finite element solutions for Helmholtz equation
(\ref{homohelm}) (with a suitable boundary condition) do not decrease when $h$ decreases unless $h\omega\rightarrow 0$ or the number of basis functions on every element also increases.
Let the mesh sizes $h$ satisfy a basic assumption $h\omega=O(1)$, and let the number of basis functions on every element be fixed. A natural question is whether
the previous property can be kept for Helmholtz equations with large wave numbers? The answer is positive definite, provided that basis functions are constructed carefully.
In \cite{Betcke2012, Giladi2001, Fang2017, Fang2016}, plane wave type methods based on the {\it geometrical optics ansatz} were proposed to
remove the pollution effect in the sense that the $L^2$ errors of the approximate finite element solutions decrease when $\omega$ increases or $h$ decreases (assuming $h\omega=const.$).
 The key idea is to define each basis function as the form $p({\bf r})e^{i\kappa{\bf d}\cdot{\bf r}}$, where $p({\bf r})$ is a polynomial and ${\bf d}$ denotes a local approximation of a wave direction
 vector determined by the geometric optics ansatz.
In \cite{Betcke2012}, assuming that a good local approximation ${\bf d}$ of every wave direction is known (for example, it has been computed by the ray tracing technique),
the best $L^2$ approximation of the resulting finite element space was derived. An important observation made in \cite{Fang2017,Fang2016} is that the wave directions are independent of $\omega$, so
every local approximation ${\bf d}$ can be preliminarily computed by the numerical micro-local technique (NMLA) \cite{Benamou2004} based on approximate solutions of an auxiliary low-frequency problem. The approximate wave directions are incorporated as plane wave basis functions of a finite element space that is applied to the discretization of the considered high-frequency problem \cite{Fang2017, Fang2016}. It has been shown in \cite{Fang2017,Fang2016} that the $L^2$ errors of the approximate solutions are ${O}(\omega^{-1/2})$ for the case with multiple waves as $\omega\rightarrow\infty$. It is clear that the $L^2$ errors ${O}(\omega^{-1/2})$ of the approximate solutions are not small unless $\omega$ is very large, which limits the applications of the method.
%The numerical micro-local technique \cite{Benamou2004} for computing local wave directions plays a key role in \cite{Fang2017}.
The unsatisfactory $L^2$ errors come from the low accuracy ${O}(\omega^{-1/2})$ of the approximate wave directions computed by the NMLA.

In this paper, inspired by the ideas in \cite{Benamou2004}, we design a new algorithm for computing local wave directions to improve their accuracy. We show that the accuracy of the computed local wave direction ${\bf d}$
as well as the best $L^2$ approximate errors of the resulting plane wave space can achieve ${O}(\omega^{-1})$ ($={O}(h)$ if $h\omega=const.$), which is much smaller than ${O}(\omega^{-1/2})$ for a large $\omega$.
%solutions of the Helmholtz equation (\ref{homohelm}) (with a variable wave number $\kappa$) possess much better $L^2$ error  ${O}(\omega^{-1})$.
Then, by combining the ideas proposed in \cite{Huqy2018}, we apply the constructed plane wave spaces to the discretization of the nonhomogeneous Helmholtz equation (\ref{nonhomogeneousHelm}).
We test two examples to confirm that the $L^2$ errors of the resulting approximate solutions linearly decay when the mesh size $h$ decreases or $\omega$ increases (choosing $h$ such that $h\omega=const.$).

The paper is organized as follows. In Section \ref{GOPWs}, we derive basic expressions of the new plane wave type basis functions by the geometrical optics ansatz. In section 3, we describe
algorithms for approximately computing local wave directions and investigate the accuracies of the approximate wave directions and prove the best $L^2$ error of the resulting finite element spaces
. Finally, we report some numerical results on the proposed methods in Section \ref{numerical}.

\section{Geometric optics ansatz-based plane wave basis functions}\label{GOPWs}

For a given $h>0$, we divide the domain $\Omega$ into a union of quasi-uniformly polygonal elements with the size $h$. Let $\mathcal{T}_h$ denote the resulting partition.
For convenience, we separate $\omega$ from $\kappa$: $\kappa(\mathbf{r})=\omega\sqrt{\xi(\mathbf{r})}$ with $\xi(\mathbf{r})= 1/c^2(\mathbf{r})$.

\subsection{Geometric optics ansatz}\label{goa}

If the solution of the equation (\ref{homohelm}) corresponds to a simple wave, according to the geometric optics ansatz, the solution of (\ref{homohelm}) can be expressed as
the L{\"u}neberg-Kline expansion \cite{Bouche1997}:
\begin{equation}\label{geoana}
	u(\omega,\mathbf{r}) = e^{i\omega\phi(\mathbf{r})} A(\mathbf{r}) ,
\end{equation}
where $\phi$ is called the phase function satisfying the eikonal equation
\begin{equation} \label{eik}
	|\nabla\phi(\mathbf{r})|^2 = \xi(\mathbf{r}),
\end{equation}
and $A(\mathbf{r})$ is called the amplitude function that can be written as
\begin{equation}\label{ansatz}
	A(\mathbf{r}) = \sum_{s = 0}^{\infty}\frac{A_s(\mathbf{r})}{(i\omega)^s}
\end{equation}
with $\{A_s\}_{s = 0}^{\infty}$ satisfying a recursive system of PDEs:
\begin{equation} \label{tran}
	2\nabla\phi\cdot\nabla A_s + A_s \Delta\phi = -\Delta A_{s-1}
\end{equation}
for $s = 0, 1, \cdots$, with $A_{-1} \equiv 0$.

The key features of the geometric optics ansatz are:
\begin{itemize}
	\item $\{A_s\}_{s = 0}^{\infty}$ and $\phi$ are independent of the frequency $\omega$;
	\item $\{A_s\}_{s = 0}^{\infty}$ and $\phi$ depend on $c(\mathbf{r})$ (and $f(\mathbf{r})$ if (\ref{nonhomogeneousHelm}) is considered).
\end{itemize}

When more waves are involved in the solution of the equation (\ref{homohelm}), the generic solution of (\ref{homohelm}) should be locally defined as a finite sum of terms like (\ref{geoana}). Hence, in general crossing waves, we use $N(\mathbf{r})$ to denote the number of crossing waves at the position $\mathbf{r}$ and expresse the solution of the Helmholtz equation (\ref{homohelm}) as
\begin{equation}\label{geoanaN}
	u(\omega,\mathbf{r}) = \sum_{n = 1}^{N(\mathbf{r})} u_n(\omega,\mathbf{r}),
\end{equation}
where each $u_n(\omega,\mathbf{r})$ has its ansatz form as (\ref{geoana})
\begin{equation}\label{geoanan}
	u_n(\omega,\mathbf{r}) = A_n({\bf r})e^{i\omega\phi_n(\mathbf{r})}=e^{i\omega\phi_n(\mathbf{r})}\sum_{s = 0}^{\infty}\frac{A_{n,s}(\mathbf{r})}{(i\omega)^s}.
\end{equation}
And for the $n$-wave ansatz ($n = 1,\cdots,N(\mathbf{r})$), the $\omega$-independent phase function $\phi_n(\mathbf{r})$ and $\{A_{n,s}(\mathbf{r})\}_{s = 0}^{\infty}$ satisfy the eikonal equation (\ref{eik}) and the corresponding system (\ref{tran}) respectively.

\subsection{Construction of plane wave type basis functions} Recall that the solution of (\ref{homohelm}) can be written as (see (\ref{geoanaN}) and (\ref{geoanan}))
$$ u(\omega,\mathbf{r}) = \sum_{n = 1}^{N(\mathbf{r})} u_n(\omega, \mathbf{r}), \quad u_n(\omega, \mathbf{r}) = A_n(\omega,\mathbf{r})e^{i\omega\phi_n(\mathbf{r})} $$
and the phase functions $\{\phi_n(\mathbf{r})\}_{n=1}^{N(\mathbf{r})}$ are independent of $\omega$. If the phase functions $\phi_n({\bf r})$ are
known, then we need only to determine the amplitude functions $A_n(\omega, {\bf r})$, which can be approximated by finite element functions with smaller degrees of freedom.

We need only to consider the case with a single wave. For a sufficiently large $\omega$, by (\ref{geoana}) and (\ref{ansatz}) we have
$$ u(\omega, {\bf r})=A_0({\bf r})e^{i\omega\phi({\bf r})}+O(\omega^{-1}). $$
Motivated by the above expression, we define a plane wave type basis function as $\varphi(\mathbf{r})= p(\mathbf{r})e^{i\omega\tau(\mathbf{r})}$, where $\tau$ is a real polynomial approximately satisfying (\ref{eik}) and $a$ is  a complex polynomial defined by $\tau(\mathbf{r})$.

We consider a generic element $K_0$ with the diameter $h$ and the barycenter ${\bf r}_0=(x_0,y_0)$. By the Taylor formula, $\phi$ can be written as
$$ \phi({\bf r})=\phi({\bf r}_0)+(\nabla \phi)({\bf r}_0)\cdot({\bf r}-{\bf r}_0)+O(|{\bf r}-{\bf r}_0|^2). $$
Then
$$ u(\omega, {\bf r})=A_0({\bf r})e^{i\omega\phi({\bf r}_0)}e^{i\omega(\nabla \phi)({\bf r}_0)\cdot({\bf r}-{\bf r}_0)+O(\omega h^2)}+O(\omega^{-1}). $$
It follows by (\ref{eik}) that
$$ |(\nabla \phi)({\bf r}_0)|^2=\xi({\bf r}_0), $$
which means that $(\nabla \phi)({\bf r}_0)$ can be written as  $(\nabla \phi)({\bf r}_0)=\sqrt{\xi({\bf r}_0)}(\cos\theta, \sin\theta)$ with some unknown direction angle $\theta\in [0,\pi)$.
Let $\theta_h$ be a good approximation of $\theta$. A natural idea is to choose
$$ \tau_h({\bf r})=\sqrt{\xi({\bf r}_0)}(\cos\theta_h~ \sin\theta_h)\cdot({\bf r}-{\bf r}_0),$$
which can be regarded as an approximation of $\phi({\bf r})-\phi({\bf r}_0)$. It is easy to see that such $\tau_h({\bf r})$ approximately satisfies the eikonal equation in the sense that
$$  |(\nabla \tau_h)({\bf r})|^2=\xi({\bf r})+O(|{\bf r}-{\bf r}_0|),\quad {\bf r}\in K_0. $$
For this linear polynomial $\tau_h$, we look for a linear polynomial $p_h({\bf r})$ such that $p_h({\bf r})$ satisfies a similar equation with (\ref{tran}) that $A_0({\bf r})$ meets (we can omit the constant $e^{i\omega\phi({\bf r}_0)}$). Replacing $\phi$ and $A_0$ in (\ref{tran}) (for $l=0$) by $\tau_h$ and $p_h$ respectively and using the fact $\Delta \tau_h=0$, we can define a linear polynomial $p_h({\bf r})$ by
$$ \nabla \tau_h\cdot \nabla p_h=0,\quad \mbox{on}~~K_0. $$
%In the next section we will find such a linear polynomial $p({\bf r})$ to approximate $A_0({\bf r})$.
We have two independent choices of $p_h({\bf r})$ satisfying the above equation, namely,
$$ p^{\tau_h}_{1}({\bf r})=1+(\sin\theta_h~-\cos\theta_h)\cdot({\bf r}-{\bf r}_0)\quad \mbox{and}\quad p^{\tau_h}_{2}({\bf r})=1-(\sin\theta_h~-\cos\theta_h)\cdot({\bf r}-{\bf r}_0).$$
They correspond to two plane wave basis functions
$$ \varphi_1({\bf r})=p^{\tau_h}_{1}({\bf r})e^{i\omega\tau_h({\bf r})}\quad\mbox{and}\quad \varphi_2({\bf r})=p^{\tau_h}_{2}({\bf r})e^{i\omega\tau_h({\bf r})}. $$

For the case with more waves, we can similarly define plane wave-type basis functions.
%For this case, the eikonal equation has more solutions.
For an element $K_0$ with the barycenter ${\bf r}_0$, define
\begin{equation*}
	{\mathbf{d}}_n := \frac{\nabla\phi_n(\mathbf{r}_0)}{|\nabla\phi_n(\mathbf{r}_0)|} = c(\mathbf{r}_0)\nabla\phi_n(\mathbf{r}_0) \quad (n=1,\cdots, N(\mathbf{r}_0)).
\end{equation*}	
By the eikonal equation, we have $|{\mathbf{d}}_n|=1$ and so it can be written as ${\bf d}_n=(\cos\theta_{n}~\cos\theta_{n})^t$ with $\theta_{n}\in [0,\pi)$. We call ${\mathbf{d}}_n$ and $\theta_{n}$
$n$-th ray direction (or wave direction) of the wave fronts at $\mathbf{r}_0$ and the direction angle of the $n$-th ray, respectively.

Let $\theta_{h,n}$ be an approximation of $\theta_{n}$. Define
$$ {\bf d}_{h,n}=(\cos\theta_{h,n}~\cos\theta_{h,n})^t\quad \mbox{and}\quad \tau_{h,n}=\sqrt{\xi({\bf r}_0)}{\bf d}_{h,n}\cdot({\bf r}-{\bf r}_0), $$
which can be regarded as an approximation of $\phi_n(\mathbf{r})-\phi_n(\mathbf{r}_0)$.
Define
$$ p_{1}^{\tau_{h,n}}({\bf r})=1+{\bf d}^{\bot}_{h,n}\cdot({\bf r}-{\bf r}_0)\quad\mbox{and}\quad p_{h}^{\tau_{h,n}}({\bf r})=1-{\bf d}^{\bot}_{h,n}\cdot({\bf r}-{\bf r}_0) $$
which satisfy the equation
$$ \nabla\tau_{h,n}\cdot\nabla p_j^{\tau_{h,n}}=0,\quad\mbox{on}~~K_0. $$
Then we define $2N({\bf r}_0)$ plane wave basis functions on $K_0$ as follows
$$ \varphi_{n,j}(\omega,{\bf r})=p_j^{\tau_{h,n}}({\bf r})e^{i\omega\tau_{h,n}({\bf r})},\quad n=1,\cdots,N({\bf r}_0); ~j=1,2\quad({\bf r}\in K_0). $$
The core task of this article is to compute approximate direction angles $\theta_{h,n}$ ($n=1,\cdots, N({\bf r}_0)$).

\begin{remark} In most applications, there are only several rays , i.e., $N({\bf r}_0)$ is small. Then the number of the local
basis functions $\{\varphi_{n,j}\}$ is less than that of the $p-version$ of the plane wave methods, so the plane wave method with the proposed basis functions is cheaper than the $p-version$ of the plane wave methods for (\ref{homohelm}).
In the existing works \cite{Betcke2012} and \cite{Fang2017}, the factor $p_h({\bf r})$ was directly chosen as complete linear polynomials, which corresponds to three independent basis functions. Here we have used the
geometric optics ansatz (\ref{tran}) to reduce the number of basis functions on each element.
\end{remark}

\section{An adaptive plane wave method based on approximate direction angles}\label{raybasedpw}

Since the ray angles $\{\theta_{n}\}_{n = 1}^{N(\mathbf{r}_0)}$ described in Subsection 2.2
are not known in applications, we hope to find a cheap way to compute good approximations $\{\theta_{h,n}\}_{n = 1}^{N(\mathbf{r}_0)}$ of them for a large $\omega$ .
%In this section, we devoted to development of an adaptive-type plane wave method for homogeneous Helmholtz equation (\ref{homohelm}) with high-frequency ($\omega$ is large).
%Let us first describe the main ideas (refer to \cite{Benamou2004} and  \cite{Fang2016}).
An important observation made in \cite{Fang2017,Fang2016} is that the ray angles $\{\theta_{n}\}_{n = 1}^{N(\mathbf{r}_0)}$ are independent of $\omega$, so we can use an approximate solution of low-frequency problem
to compute $\{\theta_{h,n}\}_{n = 1}^{N(\mathbf{r}_0)}$.

 Choosing positive numbers $\tilde{\omega}\ll\omega$, and consider auxiliary Helmholtz equations
\begin{equation}\label{homohelm1}
\left\{\begin{aligned}
	&(\Delta +\tilde{\omega}^2\xi(\mathbf{r})) \tilde{u}(\tilde{\omega},\mathbf{r}) =0, \quad \mathbf{r} = (x, y) \in\Omega, \\
	&(\partial_{\mathbf{n}} + i\tilde{\omega}\sqrt{\xi(\mathbf{r})}) u(\tilde{\omega},\mathbf{r}) = g(\tilde{\omega}, \mathbf{r}), \quad \mathbf{r}\in\partial\Omega.
	\end{aligned}
	\right.
\end{equation}
The solution of (\ref{homohelm1}) can be written as
\begin{equation}\label{lowsolu} \tilde{u}(\tilde{\omega},\mathbf{r}) = \sum_{n = 1}^{N(\mathbf{r})} \tilde{u}_n(\tilde{\omega}, \mathbf{r}), \quad \tilde{u}_n(\tilde{\omega}, \mathbf{r}) = \tilde{A}_n(\tilde{\omega},\mathbf{r})e^{i\tilde{\omega}\phi_n(\mathbf{r})}. \end{equation}
Since $\tilde{\omega}\ll\omega$ (for example, $\tilde{\omega}=\sqrt{\omega}$), the Helmholtz equation (\ref{homohelm1}) can be numerically solved more cheaper than
the original equation (\ref{homohelm}). Let $u^{\tilde{\omega}}_h$ denote an approximate solution of (\ref{homohelm1}).

By using the approximation $u^{\tilde{\omega}}_h$, one can
probe a good approximation $\mathbf{d}_{h,n}=(\cos\theta_{h,n}, \sin\theta_{h,n})$ of the $n$-th ray direction $\mathbf{d}_n$ by the method proposed in \cite{Benamou2004}.
However, the accuracy of the resulting approximation is unsatisfactory.
%For an element $K_0$, let $\tau_{h,n}({\bf r})$ and $a^{\tau_{h,n}}_{j}({\bf r})$ ($j=1,2$) be the polynomials determined as in Subsection 2.2 by replacing the plane wave direction ${\bf d}_n$ with $\mathbf{d}_{h,n}$.  Define %the local plane wave space as
%$$ V_{ray}(K_0)=\{\sum_{n=1}^{N(\mathbf{r}_0)}a^{\tau_{h,n}}_{j}({\bf r})e^{i\omega\tau_{h,n}(\mathbf{r})}:~j=1,2\}. $$
In this section, inspired by the ideas in \cite{Benamou2004}, we will use a different method from \cite{Benamou2004} to probe approximate
ray angles $\{\theta_{h,n}\}_{n = 1}^{N(\mathbf{r}_0)}$.

Let us first introduce some common notations repeatedly used in this section.

We use $\nabla_{\mathbf{r}}$ to denote the space gradient operator with respect to the $(x,y)$ coordinate and $\nabla^2_{\mathbf{r}}$ as the Hessen operator with respect to the space variable $\mathbf{r}$.
If the applied variable only depends on space, we also simplify $\nabla_{\mathbf{r}}$ and $\nabla^2_{\mathbf{r}}$ as $\nabla$ and $\nabla^2$ respectively. Furthermore, denote
\begin{equation*}
	\nabla_{\mathbf{r}}u(\mathbf{r}_0) := (\nabla_{\mathbf{r}}u(\mathbf{r}))|_{\mathbf{r} = \mathbf{r}_0}.
\end{equation*}
Consider a reference point $\mathbf{r}_0$. For $0<\rho\ll 1$, the circle neighborhood of $\mathbf{r}_0$ with radio $\rho$ is denoted by
$$ O(\mathbf{r}_0,\rho):=\{\mathbf{r}:~|\mathbf{r} - \mathbf{r}_0|<\rho\}=\{\mathbf{r}:~\mathbf{r}=\mathbf{r}_0+ r(\cos\theta,\sin\theta),~r<\rho;0\leq\theta\leq 2\pi\}. $$

We make the following assumption on $\Omega$.

\begin{assumption}\label{mediumregion}
	We assume that $\Omega$ can be decomposed into several non-overlapping simply connected polygon region. On each region, there is only one kind of medium and $\xi$ is assumed to be sufficiently smooth. Hence $N(\mathbf{r})$ is a constant on each region.
\end{assumption}

In this section we give an adaptive ray learning method for general wave solution.

\subsection{Approximate ray angles $\{\theta_n\}$ determined from an analytic solution of the low-frequency problem}~~

%In case of multiple waves, directly using the NMLA method on the multiple wave solution $u(\omega, \mathbf{r})$ to compute the angle of the ray directions only have an error estimation order of $\omega^{-1/2}$, which greatly affects the approximation error of the adaptive ray-based methods \cite{}.

%In this section, we provide an efficient way of improving the approximation of learning ray directions for multiple waves.

For a point ${\bf r}\in O({\bf r}_0,\rho)$, we write ${\bf r}=\mathbf{r}_0 + r {\bf d}_{\theta}$ with ${\bf d}_{\theta}=(\cos\theta,\sin\theta)$. A circle centered at $\mathbf{r}_0$ with radius $\rho$ is denoted by
$$ S_{\rho}(\mathbf{r}_0) := \{\mathbf{r}:~|\mathbf{r} - \mathbf{r}_0| = \rho\}=\{\mathbf{r}:~\mathbf{r}=\mathbf{r}_0+ \rho{\bf d}_{\theta},~0\leq\theta< 2\pi\}. $$

We first recall the NMLA method developed in \cite{Benamou2004}, which can be used to compute ray directions with low accuracy.

{\bf NMLA.} Set $\tilde{\omega} = \sqrt{\omega}$ and let $\tilde{u}(\tilde{\omega}, {\bf r})$ be the solution of the low-frequency problem (\ref{homohelm1}) ($\tilde{u}(\tilde{\omega}, {\bf r})$ was represented by (\ref{lowsolu})). Choosing $r_0>0$ such that $\tilde{\omega} r_0^2 = \mathcal{O}(1)$. Define an auxiliary function (which was called impedance quantity)
\begin{equation}
	U_{\tilde{\omega}}(\theta) := (1+\frac{c(\mathbf{r}_0)}{i\tilde{\omega}}\partial_r) \tilde{u}(\tilde{\omega}, \mathbf{r}_0+ r_0\mathbf{d}_{\theta}),
%\quad \alpha_0 = \frac{1}{c(\mathbf{r}_0)}\tilde{\omega} r_0,
\end{equation}
which removes any possible ambiguity due to resonance \cite{Fang2017} and improves the robustness to noise for solutions of the Helmholtz equation.

The NMLA method sample the impedance quantity $U_{\tilde{\omega}}(\theta)$ on the circle $S_{r_0}(\mathbf{r}_0)$. To this end,
let $(\mathcal{F}U_{\tilde{\omega}})_{\ell}$ denote the $\ell$-th Fourier coefficient of $U_{\tilde{\omega}}$, namely,
\begin{equation*}
	(\mathcal{F}U_{\tilde{\omega}})_{\ell} = \frac{1}{2\pi} \int_{0}^{2\pi} U_{\tilde{\omega}}(\theta)e^{-i\ell\theta} d\theta.
\end{equation*}
Set $\tilde{\omega}_0=\frac{1}{c(\mathbf{r}_0)}\tilde{\omega} r_0$. Then apply the filtering operator $\mathcal{B}$ to the impedance quantity $U_{\tilde{\omega}}$
%and obtain
\begin{equation} \label{filter}
		\mathcal{B}U_{\tilde{\omega}}(\theta) = \frac{1}{2M_{\tilde{\omega}_0} + 1} \sum_{\ell = -M_{\tilde{\omega}_0}}^{M_{\tilde{\omega}_0}} \frac{(\mathcal{F}U_{\tilde{\omega}})_{\ell} e^{i\ell\theta}}{i^{\ell}(J_{\ell}(\tilde{\omega}_0)
 - iJ_{\ell}^{\prime}(\tilde{\omega}_0))}
\end{equation}
where $M_{\tilde{\omega}_0} = \max(1, [\tilde{\omega}_0], [\tilde{\omega}_0+(\tilde{\omega}_0^{1/3}-2.5)])$. Notice that the number of waves $N(\mathbf{r}_0)$ is just the number of sharp peaks in the graph of function $\mathcal{B} U_{\tilde{\omega}}(\theta)$.

Define $\tilde{B}_n(\tilde{\omega},\mathbf{r}) = \tilde{A}_n(\tilde{\omega},\mathbf{r}) e^{i\tilde{\omega}\phi_n(\mathbf{r}_0)}$.
It was shown in \cite{Benamou2011} that
 \begin{equation}\label{filter}
	\mathcal{B} U_{\tilde{\omega}}(\theta) = \sum_{n=1}^{N(\mathbf{r}_0)} \tilde{B}_n(\tilde{\omega},\mathbf{r}_0) S_{M_{\tilde{\omega}_0}}(\theta-\theta_{n}),
\end{equation}
where $S_{k}(\theta)=\frac{\sin ([2k+1] \theta / 2)}{[2 k+1] \sin (\theta / 2)}$. As a consequence, when $\tilde{\omega}_0\rightarrow\infty$ we have
\begin{equation} \label{filter0}
	\mathcal{B} U_{\tilde{\omega}}(\theta) =
	\left\{
	\begin{aligned}
		& \tilde{B}_n(\tilde{\omega},\mathbf{r}_0), \quad \text{if\ } \theta = \theta_n, \\
		& 0, \quad \text{otherwise}.
	\end{aligned}
	\right.
\end{equation}
For a sufficiently large $\tilde{\omega}$, we compute $\theta_n^*$ and $\tilde{B}^{*}_n(\tilde{\omega},\mathbf{r}_0)$ by
$$ \tilde{B}^{*}_n(\tilde{\omega},\mathbf{r}_0) = \mathcal{B} U_{\tilde{\omega}}(\theta_n^*)=\max\limits_{\theta\in (0,2\pi]}\mathcal{B} U_{\tilde{\omega}}(\theta). $$
Then (see \cite{Benamou2011})
 \begin{equation}
|\theta_n^{*}-\theta_n| \sim \mathcal{O}(\tilde{\omega}^{-1/2}), \quad |\tilde{B}^{*}_n(\tilde{\omega},\mathbf{r}_0)-\tilde{B}_n(\tilde{\omega},\mathbf{r}_0)| \sim \mathcal{O}(\tilde{\omega}^{-1/2}).\label{ray_angle}
\end{equation}

As we will see, the accuracy of the approximate ray angles are lower than the expected accuracy ($\mathcal{O}(\tilde{\omega}^{-2})$), which determines the approximation error of the ray-based FEM method.

To improve the accuracy of the approximate ray angle, we propose a post-processing method where the data obtained from the NMLA is used as the initial data for a routine that tries to fit some information of the wave $u(\tilde{\omega}, \mathbf{r})$.

{\bf Post-processing.}  We need to investigate how to use $\tilde{u}$ (if it is known) to learn more exact ray angles.
Take a sampling circle $S_{r_1}(\mathbf{r}_0)$ with $r_1$ satisfying $\tilde{\omega} r_1 = \mathcal{O}(1)$. Define two dual impedance quantities on the circle $S_{r_1}(\mathbf{r}_0)$ as
\begin{equation}\label{secapp}
	\left\{\begin{aligned}
	U_{\tilde{\omega}}^{+}(\theta)&=(1+\frac{c(\mathbf{r}_0)}{i\tilde{\omega}}\partial_r)\tilde{u}(\tilde{\omega},\mathbf{r}_0 + r_1\mathbf{d}_{\theta}),\\
	U_{\tilde{\omega}}^{-}(\theta)&=(1-\frac{c(\mathbf{r}_0)}{i\tilde{\omega}}\partial_r)\tilde{u}(\tilde{\omega},\mathbf{r}_0 + r_1\mathbf{d}_{\theta}).
	\end{aligned}\right.
\end{equation}

Set $\tilde{\omega}_1 = \frac{1}{c(\mathbf{r}_0)}\tilde{\omega} r_1$. Choosing a positive integer $M_{\tilde{\omega}_1}$ satisfying $M_{\tilde{\omega}_1} \geq \tilde{\omega}_1$.
For each integer $l$ satisfying $|l|\leq M_{\tilde{\omega}_1}$, define a sampling quantities by
%\begin{equation}\label{postl}
$$	\widetilde{U}_{\ell}={1\over i^{\ell}}\bigg(\frac{[\mathcal{F} U_{\tilde{\omega}}^{+}(\theta)]_{\ell}} {\frac{i\tilde{\omega}_1}{2}(J_{\ell}(\tilde{\omega}_1) - iJ_{\ell}^{'}(\tilde{\omega}_1)) + J_{\ell}(\tilde{\omega}_1)} - \frac{[\mathcal{F} U_{\tilde{\omega}}^{-}(\theta)]_{\ell}} {\frac{i\tilde{\omega}_1}{2}(J_{\ell}(\tilde{\omega}_1) + iJ_{\ell}^{'}(\tilde{\omega}_1)) - J_{\ell}(\tilde{\omega}_1)}\bigg).
$$
%\end{equation}

Before describing the post-processing, we give an expansion of $\widetilde{U}_{\ell}$.
Set
$$ \mathbf{a}_{-1}= (\frac{1}{2}, -\frac{i}{2})^T,\quad \mathbf{a}_1 = (\frac{1}{2}, \frac{i}{2})^T. $$
For $n=1,\cdots,N(\mathbf{r}_0)$;
$j=-1,1$, define the parameters
$$ b_{n,0}=\tilde{B}_n(\tilde{\omega},\mathbf{r}_0),\quad b_{n,j}=r_1\mathbf{a}_{-j}\cdot \nabla_{\mathbf{r}} \tilde{B}_n(\tilde{\omega},\mathbf{r}_0) $$
and
$$ b_{n,2j}=r_1c(\mathbf{r}_0)\tilde{B}_n(\tilde{\omega},\mathbf{r}_0)\mathbf{a}_{-j}^T\nabla^2\phi_n(\mathbf{r}_0)\mathbf{a}_{-j}, $$
where $\tilde{B}_n(\tilde{\omega},\mathbf{r})=\tilde{A}_n(\tilde{\omega},\mathbf{r})e^{i\omega\phi_n(\mathbf{r}_0)}$ (see the last part). Besides, for $|\ell|\leq M_{\tilde{\omega}_1}$ we define
$$ \mu_{\ell,0}=\frac{2J_{\ell}^2(\tilde{\omega}_1)}{a_{\ell}^{+} a_{\ell}^{-}},~~
\mu_{\ell,j}=j \frac{\tilde{\omega}_1(J_{\ell+j}^2(\tilde{\omega}_1)+J_{\ell}^2(\tilde{\omega}_1)) - 2(\ell+j)J_{\ell+j}(\tilde{\omega}_1)J_{\ell}(\tilde{\omega}_1)}{a_{\ell}^{+} a_{\ell}^{-}} \quad (j=\pm 1)$$
and
$$ \mu_{\ell,2j}=ij(l+j)\tilde{\omega}_1 \frac{J^2_{\ell+j}(\tilde{\omega}_1) - J_{\ell+2j}(\tilde{\omega}_1)J_{\ell}(\tilde{\omega}_1)}{a_{\ell}^{+} a_{\ell}^{-}}\quad (j=\pm 1). $$
Here
$$ a_{\ell}^{+}=\frac{i\tilde{\omega}_1}{2}(J_{\ell}(\tilde{\omega}_1) - iJ_{\ell}^{'}(\tilde{\omega}_1)) + J_{\ell}(\tilde{\omega}_1),~~a_{\ell}^{-}=\frac{i\tilde{\omega}_1}{2}(J_{\ell}(\tilde{\omega}_1) + iJ_{\ell}^{'}(\tilde{\omega}_1)) - J_{\ell}(\tilde{\omega}_1). $$

\begin{lemma} For each integer $l$ satisfying $|l|\leq M_{\tilde{\omega}_1}$, the number $\widetilde{U}_{\ell}$ has the expansion for large $\tilde{\omega}$
\begin{equation}
\widetilde{U}_{\ell}\label{multisample}
= \sum_{n=1}^{N(\mathbf{r}_0)}\sum_{k=-2}^{2}b_{n,k}\mu_{\ell,k}e^{-i(l+k)\theta_n}.
\end{equation}
%\begin{eqnarray}
%	\widetilde{U}_{\ell}\label{multisample}
%	& = &\sum_{n=1}^{N(\mathbf{r}_0)} \bigg(\mu_{0,\ell}b_{n,0} e^{-i\ell\theta_n}+\sum_{j=-1,1}
%    \mu_{j,\ell}b_{n,j} e^{-i(\ell+j)\theta_n} \cr
%	&+&\sum_{j=-1,1}\mu_{2j,\ell}b_{n,2j} e^{-i(l+2j)\theta_n}\bigg)+\mathcal{O}(\tilde{\omega}^{-2}),
%\end{eqnarray}
%where $\mu_{j,\ell}$ ($|\ell|\leq M_{\tilde{\omega}_1}$, $|j|\leq 2$) are
%constants that are independent of $\tilde{u}$ and can be explicitly computed by $\tilde{\omega}_1$.
\end{lemma}

\begin{proof} By the Taylor expansion on $r\in (0,\tilde{\omega}^{-1}]$, we have
\begin{eqnarray*}
&&\tilde{A}_n(\tilde{\omega}, \mathbf{r})
	 = \tilde{A}_n(\tilde{\omega},\mathbf{r}_0 + r\mathbf{d}_{\theta}) = \tilde{A}_n(\tilde{\omega},\mathbf{r}_0) + r \nabla_{\mathbf{r}} \tilde{A}_n(\tilde{\omega},\mathbf{r}_0) \cdot \mathbf{d}_{\theta} + \mathcal{O}(r^2), \cr
&&\phi_n(\mathbf{r})
	 = \phi_n(\mathbf{r}_0 + r\mathbf{d}_{\theta}) = \phi_n(\mathbf{r}_0) + r \frac{\mathbf{d}_n \cdot \mathbf{d}_{\theta}}{c(\mathbf{r}_0)}+ \frac{r^2}{2}\mathbf{d}_{\theta}^T\nabla^2\phi_n(\mathbf{r}_0)\mathbf{d}_{\theta} + \mathcal{O}(r^3).
\end{eqnarray*}
By the definition of $\tilde{u}$ and using the above expansions, we deduce that
\begin{eqnarray*}
\tilde{u}(\tilde{\omega},\mathbf{r}_0 + r\mathbf{d}_{\theta}) &=& \sum_{n=1}^{N(\mathbf{r}_0)} \big(\tilde{B}_n(\tilde{\omega},\mathbf{r}_0) + r \nabla_{\mathbf{r}} \tilde{B}_n(\tilde{\omega},\mathbf{r}_0) \cdot \mathbf{d}_{\theta} \cr
&+&i\tilde{\omega}\frac{r^2}{2} \tilde{B}_n(\tilde{\omega},\mathbf{r}_0)\mathbf{d}_{\theta}^T\nabla^2\phi_n(\mathbf{r}_0)\mathbf{d}_{\theta}\big)e^{i \frac{\tilde{\omega} r}{c(\mathbf{r}_0)} \mathbf{d}_n \cdot \mathbf{d}_{\theta}}+\mathcal{O}(\tilde{\omega}^{-2}).
\end{eqnarray*}
Then we have
\begin{eqnarray*}
%\label{imp1}
	&&U_{\tilde{\omega}}^{+}(\theta)
	= \sum_{n=1}^{N(\mathbf{r}_0)}\bigg\{\tilde{B}_n(\tilde{\omega},\mathbf{r}_0)(1+\mathbf{d}_n \cdot \mathbf{d}_{\theta}) + r_1\nabla_{\mathbf{r}} \tilde{B}_n(\tilde{\omega},\mathbf{r}_0)\cdot \mathbf{d}_{\theta} (1+\mathbf{d}_n \cdot \mathbf{d}_{\theta} + \frac{1}{i\tilde{\omega}_1}) \cr
	&&+ r_1\tilde{B}_n(\tilde{\omega},\mathbf{r}_0)\mathbf{d}_{\theta}^T\nabla^2\phi_n(\mathbf{r}_0)\mathbf{d}_{\theta} [ \frac{i\tilde{\omega} r_1}{2}(1+\mathbf{d}_n \cdot \mathbf{d}_{\theta}) + c(\mathbf{r}_0) ] \bigg\}e^{i\tilde{\omega}_1 \mathbf{d}_n \cdot \mathbf{d}_{\theta}}+\mathcal{O}(\tilde{\omega}^{-2})
\end{eqnarray*}
and
\begin{eqnarray*}
%\label{imp2}
	&&U_{\tilde{\omega}}^{-}(\theta)= \sum_{n=1}^{N(\mathbf{r}_0)}\bigg\{\tilde{B}_n(\tilde{\omega},\mathbf{r}_0)(1-\mathbf{d}_n \cdot \mathbf{d}_{\theta}) + r_1\nabla_{\mathbf{r}} \tilde{B}_n(\tilde{\omega},\mathbf{r}_0)\cdot \mathbf{d}_{\theta} (1-\mathbf{d}_n \cdot \mathbf{d}_{\theta} - \frac{1}{i \tilde{\omega}_1}) \cr
	&&+ r_1\tilde{B}_n(\tilde{\omega},\mathbf{r}_0)\mathbf{d}_{\theta}^T\nabla^2\phi_n(\mathbf{r}_0)\mathbf{d}_{\theta} [ \frac{i\tilde{\omega} r_1}{2}(1-\mathbf{d}_n \cdot \mathbf{d}_{\theta}) - c(\mathbf{r}_0) ] \bigg\}e^{i\tilde{\omega}_1\mathbf{d}_n \cdot \mathbf{d}_{\theta}}+\mathcal{O}(\tilde{\omega}^{-2}).
\end{eqnarray*}
Substituting $\mathbf{d}_{\theta} = e^{i\theta}\mathbf{a}_{-1}+ e^{-i\theta}\mathbf{a}_1$ into the above two equalities and using the 2-D Jacobi-Anger expansion
$$e^{i\tilde{\omega}_1\mathbf{d}_n \cdot \mathbf{d}_{\theta}} = \sum_{\ell=-\infty}^{\infty} i^{\ell}J_{\ell}(\tilde{\omega}_1) e^{-i\ell\theta_n+i\ell\theta},$$
we can directly verify that the coefficients of two Fourier transformation are
\begin{eqnarray*}
%\label{imp1}
	&&\frac{1}{i^{\ell}}[\mathcal{F} U_{\tilde{\omega}}^{+}(\theta)]_{\ell}
	= \mathcal{O}(\tilde{\omega}^{-2}) + \sum_{n=1}^{N(\mathbf{r}_0)}\bigg\{b_{n,0} (J_{\ell}(\tilde{\omega}_1) - iJ^{'}_{\ell}(\tilde{\omega}_1))e^{-i\ell\theta_n} \cr
	&&+ \sum_{j = -1, 1} b_{n, j}[(1+\frac{1}{i\tilde{\omega}_1})J_{\ell+j}(\tilde{\omega}_1) - iJ^{'}_{\ell+j}(\tilde{\omega}_1)] e^{-i(\ell + j)\theta_n} \cr
	&&{\color{blue}{+ 2r_1c(\mathbf{r}_0)\tilde{B}_n(\tilde{\omega},\mathbf{r}_0)\mathbf{a}_{-1}^T\nabla^2\phi_n(\mathbf{r}_0)\mathbf{a}_{1}  [ (\frac{i\tilde{\omega}_1}{2} + 1)J_{\ell}(\tilde{\omega}_1) + \frac{\tilde{\omega}_1}{2} J^{'}_{\ell}(\tilde{\omega}_1)] e^{-i\ell\theta_n}}} \cr
	&&+ \sum_{j=-1, 1} b_{n, 2j}[(\frac{i\tilde{\omega}_1}{2} + 1)J_{\ell+2j}(\tilde{\omega}_1) + \frac{\tilde{\omega}_1}{2}J^{'}_{\ell+2j}(\tilde{\omega}_1)] e^{-i(\ell+2j)\theta_n}\bigg\}
\end{eqnarray*}
and
\begin{eqnarray*}
%\label{imp2}
	&&{1\over i^{\ell}}[\mathcal{F} U_{\tilde{\omega}}^{-}(\theta)]_{\ell}
	= \mathcal{O}(\tilde{\omega}^{-2}) + \sum_{n=1}^{N(\mathbf{r}_0)}\bigg\{b_{n,0}(J_{\ell}(\tilde{\omega}_1) + iJ^{'}_{\ell}(\tilde{\omega}_1))e^{-i\ell\theta_n} \cr
	&&+ \sum_{j = -1, 1} b_{n, j}[(1-\frac{1}{i\tilde{\omega}_1})J_{\ell+j}(\tilde{\omega}_1) + iJ^{'}_{\ell+j}(\tilde{\omega}_1)] e^{-i(\ell + j)\theta_n} \cr
	&&{\color{blue}{+ 2r_1c(\mathbf{r}_0)\tilde{B}_n(\tilde{\omega},\mathbf{r}_0)\mathbf{a}_{-1}^T\nabla^2\phi_n(\mathbf{r}_0)\mathbf{a}_{1}  [ (\frac{i\tilde{\omega}_1}{2} - 1)J_{\ell}(\tilde{\omega}_1) - \frac{\tilde{\omega}_1}{2} J^{'}_{\ell}(\tilde{\omega}_1)] e^{-i\ell\theta_n}}} \cr
	&&+ \sum_{j=-1, 1} b_{n,2j}[ (\frac{i\tilde{\omega}_1}{2} - 1)J_{\ell+2j}(\tilde{\omega}_1) - \frac{\tilde{\omega}_1}{2}J^{'}_{\ell+2j}(\tilde{\omega}_1)] e^{-i(\ell+2j)\theta_n}\bigg\}.
\end{eqnarray*}
Then, by the definition of $\widetilde{U}_{\ell}$, we obtain
\begin{eqnarray}
	\widetilde{U}_{\ell}
	& = &\sum_{n=1}^{N(\mathbf{r}_0)} \bigg(b_{n,0}\mu_{0,\ell} e^{-i\ell\theta_n}+\sum_{j=-1,1}b_{n,-j}
    \mu_{j,\ell}e^{-i(\ell+j)\theta_n} \cr
	&+&\sum_{j=-1,1}b_{n,-2j}\mu_{2j,\ell}e^{-i(l+2j)\theta_n}\bigg)+\mathcal{O}(\tilde{\omega}^{-2}),\label{3.new-equality}
\end{eqnarray}
where
\begin{eqnarray*}
	\mu_{\ell,0} = \frac{J_{\ell}(\tilde{\omega}_1) - iJ^{'}_{\ell}(\tilde{\omega}_1)}{\frac{i\tilde{\omega}_1}{2}(J_{\ell}(\tilde{\omega}_1) - iJ_{\ell}^{'}(\tilde{\omega}_1)) + J_{\ell}(\tilde{\omega}_1)}
 - \frac{J_{\ell}(\tilde{\omega}_1) + iJ^{'}_{\ell}(\tilde{\omega}_1)}{\frac{i\tilde{\omega}_1}{2}(J_{\ell}(\tilde{\omega}_1) + iJ_{\ell}^{'}(\tilde{\omega}_1)) - J_{\ell}(\tilde{\omega}_1)} = \frac{2J_{\ell}^2(\tilde{\omega}_1)}{a_{\ell}^{+} a_{\ell}^{-}},
\end{eqnarray*}
\begin{eqnarray*}
	\mu_{\ell,j} &=& \frac{(1+\frac{1}{i\tilde{\omega}_1})J_{\ell+j}(\tilde{\omega}_1) - iJ^{'}_{\ell+j}(\tilde{\omega}_1)}{\frac{i\tilde{\omega}_1}{2}(J_{\ell}(\tilde{\omega}_1) - iJ_{\ell}^{'}(\tilde{\omega}_1)) + J_{\ell}(\tilde{\omega}_1)} - \frac{(1-\frac{1}{i\tilde{\omega}_1})J_{\ell+j}(\tilde{\omega}_1) + iJ^{'}_{\ell+j}(\tilde{\omega}_1)}{\frac{i\tilde{\omega}_1}{2}(J_{\ell}(\tilde{\omega}_1) + iJ_{\ell}^{'}(\tilde{\omega}_1)) - J_{\ell}(\tilde{\omega}_1)} \cr
	& = &\frac{-J_{\ell+j}(\tilde{\omega}_1)J_{\ell}(\tilde{\omega}_1) - \tilde{\omega}_1 J_{\ell+j}(\tilde{\omega}_1)J^{'}_{\ell}(\tilde{\omega}_1) + \tilde{\omega}_1 J^{'}_{\ell+j}(\tilde{\omega}_1)J_{\ell}(\tilde{\omega}_1)}{a_{\ell}^{+} a_{\ell}^{-}} \cr
	&= &j \frac{\tilde{\omega}_1(J_{\ell+j}^2(\tilde{\omega}_1)+J_{\ell}^2(\tilde{\omega}_1)) - 2(\ell+j)J_{\ell+j}(\tilde{\omega}_1)J_{\ell}(\tilde{\omega}_1)}{a_{\ell}^{+} a_{\ell}^{-}},
\end{eqnarray*}
\begin{eqnarray*}
\mu_{\ell,2j}& =& \frac{(\frac{i\tilde{\omega}_1}{2} + 1)J_{\ell+2j}(\tilde{\omega}_1) + \tilde{\omega}_1J^{'}_{\ell+2j}(\tilde{\omega}_1)}{\frac{i\tilde{\omega}_1}{2}(J_{\ell}(\tilde{\omega}_1) - iJ_{\ell}^{'}(\tilde{\omega}_1)) + J_{\ell}(\tilde{\omega}_1)}
- \frac{(\frac{i\tilde{\omega}_1}{2} - 1)J_{\ell+2j}(\tilde{\omega}_1) - \tilde{\omega}_1J^{'}_{\ell+2j}(\tilde{\omega}_1)}{\frac{i\tilde{\omega}_1}{2}(J_{\ell}(\tilde{\omega}_1) + iJ_{\ell}^{'}(\tilde{\omega}_1)) - J_{\ell}(\tilde{\omega}_1)} \cr
	&=& \frac{\frac{i\tilde{\omega}_1^2}{2}(J^{'}_{\ell+2j}(\tilde{\omega}_1)J_{\ell}(\tilde{\omega}_1) - J_{\ell+2j}(\tilde{\omega}_1)J^{'}_{\ell}(\tilde{\omega}_1))}{a_{\ell}^{+} a_{\ell}^{-}} \cr
	&=& ij(l+j)\tilde{\omega}_1 \frac{J^2_{\ell+j}(\tilde{\omega}_1) - J_{\ell+2j}(\tilde{\omega}_1)J_{\ell}(\tilde{\omega}_1)}{a_{\ell}^{+} a_{\ell}^{-}}.
\end{eqnarray*}
The equality (\ref{3.new-equality}) is just the desired expansion (\ref{multisample}).
\end{proof}
\begin{remark} When the Fourier transformer is applied to only one of the two impedance quantities $U_{\tilde{\omega}}^{+}(\theta)$ and $U_{\tilde{\omega}}^{-}(\theta)$, we obtain complicated nonlinear models
containing the blue terms as in the NMLA method.
%where the unknowns $\theta_n$ and $b_{n,j}$ ($n=1, \cdots, N(\mathbf{r}_0)$; $|j|\leq 2$) are coupled together.
 In order to remove the blue terms,
% decouple the unknowns,
we define new sampling quantities $\tilde{U}_l$, which can simplify the fitted model.
\end{remark}
%Considering the second order approximation in (\ref{secapp}) and letting the norms of the denominators in (\ref{postl}) tend to zero as $\ell \rightarrow \infty$,
We need to determine the $6N(\mathbf{r}_0)$ unknowns $\theta_n$ and $b_{n,j}$ in (\ref{multisample}) $(|j|\leq 2; n=1, \cdots, N(\mathbf{r}_0))$.
Thus at least $6N(\mathbf{r}_0)$ equations are needed therein, which implies that $M_{\tilde{\omega}_1}\geq 3N(\mathbf{r}_0)$. Therefore the best choice
would be $\tilde{\omega}_1=M_{\tilde{\omega}_1}= 3N(\mathbf{r}_0)$, which means that the sampling radius $r_1$ can be chosen as $r_1=3N(\mathbf{r}_0)c(\mathbf{r}_0)/\tilde{\omega}$.
%Therefore the best choice would be $|\ell| \leq 3N(\mathbf{r}_0)$
%, which means that the sampling radius $r_1$ satisfies $\omega r_1 \geq 3N(\mathbf{r}_0)c(\mathbf{r}_0)$ (thus $M_{\tilde{\omega}_1}\geq 3N(\mathbf{r}_0)$).

Let $\vartheta$ denote the column vector composed of $\theta_n$ and $b_{n,j}$ $(|j|\leq 2; n=1, \cdots, N(\mathbf{r}_0))$. Define the function
$$
%G_{\ell}(\vartheta)=\sum_{n=1}^{N(\mathbf{r}_0)} \bigg(\mu_{0, \ell}b_{n,0}e^{-i\ell\theta_n}
%+\sum_{j=-1,1}\mu_{j,\ell}b_{n,j}e^{-i(\ell+j)\theta_n}
%+\sum_{j=-1,1}\mu_{2j,\ell}b_{n,2j}e^{-i(l+2j)\theta_n}\bigg).
G_{\ell}(\vartheta)=\sum_{n=1}^{N(\mathbf{r}_0)}\sum_{k=-2}^2\mu_{\ell,k}b_{n,k}e^{-i(\ell+k)\theta_n}.
$$
Then, from (\ref{multisample}) we have
\begin{equation}
\widetilde{U}_{\ell}=G_{\ell}(\vartheta)+\mathcal{O}(\tilde{\omega}^{-2}). \label{4.new1}
\end{equation}
Although the number $\widetilde{U}_{\ell}$ can be computed, the vector $\vartheta$ is unknown.

By introducing $6N(\mathbf{r}_0)$ parameters $\{\bar{\theta}_n\}$ and
$\{\chi_{n,j}\}$, which corresponds to the parameters $\theta_n$ and $b_{n,j}$, let $\lambda$ denote the vector composed of
$\{\bar{\theta}_n\}$ and $\{\chi_{n,j}\}$ $(|j|\leq 2; n=1, \cdots, N(\mathbf{r}_0))$, and define
$$
%G_{\ell}(\lambda)=\sum_{n=1}^{N(\mathbf{r}_0)} \bigg(\mu_{0, \ell}\chi_{n,0}e^{-i\ell\bar{\theta}_n}
%+\sum_{j=-1,1}\mu_{j,\ell}\chi_{n,j}e^{-i(\ell+j)\bar{\theta}_n}
%+\sum_{j=-1,1}\mu_{2j,\ell}\chi_{n,2j}e^{-i(l+2j)\bar{\theta}_n}\bigg).
G_{\ell}(\lambda)=\sum_{n=1}^{N(\mathbf{r}_0)}\sum_{k=-2}^2\mu_{\ell,k}\chi_{n,k}e^{-i(\ell+k)\bar{\theta}_n}.
$$
Define
%the residuals
%\begin{equation}
%	\varepsilon_{\ell}(\lambda):= \widetilde{U}_{\ell}-G_{\ell}(\lambda) \label{4.new2}
%\end{equation} and
the functional ($M_{\tilde{\omega}_1}=3N({\bf r}_0)$)
\begin{equation}
J(\lambda)=\sum\limits_{\ell=-M_{\tilde{\omega}_1}}^{M_{\tilde{\omega}_1}}|\widetilde{U}_{\ell}-G_{\ell}(\lambda)|^2. \label{4.new2}
\end{equation}
We need to minimize the functional $J(\lambda)$ to determine the unknown $\lambda$:
%Let ${\mathcal E}(\lambda)$ denote the column vector composed of all $\varepsilon_{\ell}(\lambda)$ ($\ell=-M_{\tilde{\omega}_1}, \cdots, M_{\tilde{\omega}_1}$).
%It is clear that ${\bf \varepsilon}$ depends on the parameters $\{\bar{\theta}_n\}$ and $\{\chi_{n,j}\}$, so we write it as ${\mathcal E}(\bar{\theta}_n, \chi_{n,j})$.
%We thus try to minimize the norm of ${\mathcal E}(\lambda)$ by varying the parameters,
\begin{equation}\label{nonlinearmin}
\min_{\lambda} J(\lambda).
%	\min_{\lambda} \|{\mathcal E}(\lambda)\|_2.
\end{equation}

We solve the minimization problem (\ref{nonlinearmin}) by the damped least-squares (DLS) algorithm \cite{Meyer1970}. Let $\theta_n^{\ast}$ and $\tilde{B}^{*}_n(\tilde{\omega},\mathbf{r}_0)$
be the low-accuracy ray angle and the ray amplitudes computed by NMLA. Then we use the preliminary values as starting values
\begin{equation}\label{dlsinit}
\bar{\theta}_n^0 = \theta_n^{\ast},\quad \chi_{n,0}^0 = \tilde{B}^{*}_n(\tilde{\omega},\mathbf{r}_0), \quad \chi_{n,j}^0 = 0,~~ j = -2,-1,1,2.
\end{equation}
$$ (n = 1,\cdots,N(\mathbf{r}_0))$$
Usually a few iterations gives a dramatic improvement to the accuracy of $\{\theta_n^{\ast}\}$. Let $\{\theta_n^{post}\}$ denote the ray angles generated by this post-processing method.
%Due to the second order approximation that we used in (\ref{secapp}) and the analytic results of the Fourier transformation in (\ref{postl}),  we have the following approximation property
We describe the final result as follows.
\begin{theorem}
	Given the analytical solution $u(\tilde{\omega}, \mathbf{r})$ $(\tilde{\omega} = \sqrt{\omega})$ of the low frequency problem (\ref{homohelm1}), the post-processing process (\ref{nonlinearmin}) can improve the approximation error of the ray directions to be $\mathcal{O}(\omega^{-1})$, namely,
\begin{equation}\label{theomultiray}
	|\theta_n^{post} - \theta_n|\leq C\tilde{\omega}^{-2}=C\omega^{-1} \quad (n = 1,\cdots, N(\mathbf{r}_0))
\end{equation}
for a sufficiently large $\omega$, where the constant $C$ depends only on $\{\theta_n\}$ (and $N(\mathbf{r}_0)$).
\end{theorem}

\begin{proof}  Let $\theta$ and $\theta^{post}$ denote the $N$-dimensional column vector composed of $\theta_1,\cdots,\theta_N$ and $\theta^{post}_1,\cdots,\theta^{post}_N$, respectively. Namely,
$$ \theta=\left(\begin{array}{c}\theta_1\cr\vdots\cr\theta_N\end{array}\right)\quad \mbox{and}\quad\theta^{post}=\left(\begin{array}{c}\theta^{post}_1\cr\vdots\cr\theta_N^{post}\end{array}\right). $$
%Let $\lambda^{post}$, which is a $(6N)$-dimensional column vector, denote a solution of the minimization problem (\ref{nonlinearmin}). For convenience, we write
%$$ \lambda^{post}=\left(\begin{array}{c}\theta^{post}\cr\chi^{post}\end{array}\right)\quad\mbox{with}\quad \theta^{post}=\left(\begin{array}{c}\theta^{post}_1\cr\vdots\cr\theta_n^{post}\end{array}\right). $$
Let $\vartheta^-$ denote the $(5N)$-dimensional column vector composed of $b_{n,j}$ $(|j|\leq 2; n=1, \cdots, N(\mathbf{r}_0))$. Define the $(6N)$-dimensional column vector
$$ \lambda^{post}_-=\left(\begin{array}{c}\theta^{post}\cr\vartheta^-\end{array}\right). $$
Norice that $J$ is a real-valued function, and $\theta$ and $\theta^{post}$ are real vectors.
Using the Taylor formula yields
$$ J(\vartheta)-J(\lambda^{post}_-)=\nabla_{\theta}J(\lambda^{post}_-)\cdot(\theta-\theta^{post})+{1\over 2}H_{\theta}(\xi)(\theta-\theta^{post})\cdot(\theta-\theta^{post}), $$
where $\xi$ is a mediate value between $\vartheta$ and $\lambda^{post}_-$, and $H_{\theta}(\xi)$ denotes the Hesse matrix of $J(\lambda)$ with the variable $\theta$ at the point $\xi$. Since $\nabla_{\theta} J(\lambda^{post}_-)=0$, the above equality
becomes
$$
J(\vartheta)-J(\lambda^{post}_-)={1\over 2}(\theta-\theta^{post})^tH_{\theta}(\xi)(\theta-\theta^{post}). $$
It follows by (\ref{4.new1})-(\ref{4.new2}) that $J(\vartheta)=\mathcal{O}(\tilde{\omega}^{-4})$. Then, noticing $J(\lambda^{post}_-)\geq 0$, the above equality leads to
\begin{equation}
(\theta-\theta^{post})^tH_{\theta}(\xi)(\theta-\theta^{post})\leq C\tilde{\omega}^{-4}. \label{Taylor1}
\end{equation}
It suffices to prove that
\begin{equation}
(\theta-\theta^{post})^tH_{\theta}(\xi)(\theta-\theta^{post})\geq c\|\theta-\theta^{post}\|^2. \label{3.inverse1}
\end{equation}

Let $\lambda^{post}$, which is a $(6N)$-dimensional column vector, denote a solution of the minimization problem (\ref{nonlinearmin}).
We first prove that $\lambda$ satisfies
\begin{equation}\label{iterassume}
\|\vartheta-\lambda^{post}\|_2 \leq \mathcal{O}(\tilde{\omega}^{-1/2}).
\end{equation}
Let $\lambda^{k}$ be the approximation generated by the $k$-th iteration of the DLS method for (\ref{nonlinearmin}). By (\ref{ray_angle}) and (\ref{dlsinit}), the initial value $\lambda^0$ satisfies
$$\|\vartheta-\lambda^0\|_2 \leq \mathcal{O}(\tilde{\omega}^{-1/2}).$$
The DLS method ensures that (see \cite{Meyer1970})
$$\|\vartheta - \lambda^k\|_2 \leq C\rho^k \|\vartheta-\lambda^0\|_2 \leq C\rho^k \tilde{\omega}^{-1/2}$$
with $0\leq \rho \leq 1$. Hence we get
$$ \|\vartheta-\lim_{k\rightarrow\infty}\lambda^k\|_2\leq C\tilde{\omega}^{-1/2},$$
which implies (\ref{iterassume}).

It follows by (\ref {iterassume}) that
$$ \|\vartheta-\xi\|_2 = \mathcal{O}(\tilde{\omega}^{-1}). $$
Using this, together with $J(\vartheta)=\mathcal{O}(\tilde{\omega}^{-4})$ and $r_1=\mathcal{O}(\tilde{\omega}^{-1})$, we can verify that (for a sufficiently large $\tilde{\omega}$)
$$ {\partial^2 J\over \partial \theta_n\partial \theta_k}\mid_{\xi}=\sum\limits_{\ell=-3N}^{3N}\ell^2\mu_{\ell,0}^2(b_{n,0}\overline{b_{k,0}}e^{-i\ell(\theta_n-\theta_k)}+\overline{b_{n,0}}b_{k,0}e^{i\ell(\theta_n-\theta_k)})+\mathcal{O}(\tilde{\omega}^{-1}).$$
Define $a_{n}^{(\ell)}=b_{n,0}e^{-i\ell\theta_n}$. The above formula can be written as
$$ {\partial^2 J\over \partial \theta_n\partial \theta_k}|_{\xi}=\sum\limits_{\ell=-3N}^{3N}\ell^2\mu_{\ell,0}^2(a_{n}^{(\ell)}\overline{a_{k}^{(\ell)}}+\overline{a_{n}^{(\ell)}}a_{k}^{(\ell)})+\mathcal{O}(\tilde{\omega}^{-1}).$$
For convenience, we set $\varepsilon_n=\theta_n-\theta^{post}_n$ and $\varepsilon=\theta-\theta^{post}$. Then
$$
\varepsilon^tH_{\theta}(\xi)\varepsilon=\sum\limits_{\ell=-3N}^{3N}\ell^2\mu_{\ell,0}^2\big(|\sum\limits_{n=1}^Na_{n}^{(\ell)}\varepsilon_n|^2
+|\sum\limits_{n=1}^N\overline{a_{n}^{(\ell)}}\varepsilon_n|^2\big)+\mathcal{O}(\tilde{\omega}^{-1})\|\varepsilon\|^2.
$$
Notice that $\mu^2_{\ell,0}$ has a positive lower bound independent of $\tilde{\omega}$, the above equality gives
\begin{equation}
\varepsilon^tH_{\theta}(\xi)\varepsilon\geq c\sum\limits_{\ell=1}^{N}|\sum\limits_{n=1}^Na_{n}^{(\ell)}\varepsilon_n|^2
+\mathcal{O}(\tilde{\omega}^{-1})\|\varepsilon\|^2.\label{3.inverse2}
\end{equation}

Define the vectors $\alpha_{\ell}=(e^{-i\ell\theta_1}\cdots e^{-i\ell\theta_N})$ and the matrix $A=\sum\limits_{\ell=1}^N\overline{\alpha}_{\ell}^t\alpha_{\ell}$.
%Set $a^{(\ell)}_{nk}=e^{-i\ell(\theta_n-\theta_k)}$. Let $A^{(\ell)}$ denote the matrix with the entries $\{a^{(\ell)}_{nk}\}$, and define $A=\sum\limits_{\ell=1}^NA^{(\ell)}$.
In addition, let $\Lambda$ denote the diagonal matrix with the diagonal entries $\{b_{n,0}\}$. It is easy to check that
$$ \sum\limits_{\ell=1}^{N}|\sum\limits_{n=1}^Na_{n}^{(\ell)}\varepsilon_n|^2=\sum\limits_{\ell=1}^{N}|\alpha_{\ell}(\Lambda\varepsilon)|^2=(\overline{\Lambda}\varepsilon)^tA(\Lambda\varepsilon). $$
Notice that the direction angles $\{\theta_n\}$ are different each other. Then the vectors $\alpha_1,\cdots,\alpha_N$ are linearly independent, and so the matrix $A$ is Hermitian positive definite.
Moreover, the minimal eigenvalue of $A$ is independent of $\tilde{\omega}$. Thus
$$ \sum\limits_{\ell=1}^{N}|\sum\limits_{n=1}^Na_{\ell}^{(n)}\varepsilon_n|^2\geq c\|\Lambda\varepsilon\|^2\geq c\|\varepsilon\|^2.$$
Here we have used the fact that $\{|b_{n,0}|\}$ has a lower bound independent of $\tilde{\omega}$. Then the constant $c$ depends on $\{\theta_n\}$ only. Substituting the above inequality into (\ref{3.inverse2}) yields (\ref{3.inverse1}), which, combing (\ref{Taylor1}), gives the desired result.

\end{proof}

If we replace $u(\tilde{\omega}, \mathbf{r})$ by an approximate solution, we can design the corresponding numerical method.
\subsection{{\bf Numerical method}}\label{multinum}

For a fixed large wave number $\omega$, set $\widetilde{\omega} = \sqrt{\omega}$. Using the given function $\xi$ and $g$ in problem (\ref{nonhomogeneousHelm}), we consider the following Helmholtz equation with much lower frequency
\begin{equation}\label{lowfhomoIBC}
\left\{
	\begin{aligned}
	&\Delta u+ \widetilde{\omega}^2\xi(\mathbf{r}) u = 0, \quad \text{in\ } \Omega, \\
	&(\partial_{\mathbf{n}} + i\widetilde{\omega}\sqrt{\xi(\mathbf{r})}) u = g(\widetilde{\omega}, \mathbf{r}), \quad \text{on\ } \partial\Omega,
	\end{aligned}
	\right.
\end{equation}
and use $u(\widetilde{\omega}, \mathbf{r})$ to denote its analytical solution. Based on a good approximate solution of $u(\widetilde{\omega}, \mathbf{r})$, we can compute high accuracy ray angles
by the following five steps:

%\begin{enumerate}[(2-i).] \label{multiproc}
%\item
{\bf Step 1}. Let $\mathcal{T}_{h_0}$ be a quasi-uniformly triangular partition of the domain $\Omega$ with the mesh size $h_0$ satisfying $\widetilde{\omega}h_0^2 \approx 1$.
%By {\bf Assumption \ref{mediumregion}}, we assume that the number of rays are constant on each elements.
We apply the GOPWDG method to the discretization of (\ref{lowfhomoIBC}) on $\mathcal{T}_{h_0}$, and use $u_{h_0}(\widetilde{\omega}, \mathbf{r})$ to denote the resulting approximate solution.

{\bf Step 2}. On each element of $\mathcal{T}_{h_0}$ with the barycenter $\mathbf{r}_0$, we apply the NMLA method to $u_{h_0}(\widetilde{\omega}, \mathbf{r})$. We compute the number $N_{h_0}(\mathbf{r}_0)$
of rays, ray angles $\{\theta_{{h_0},n}\}_{n=1}^{N_{h_0}(\mathbf{r}_0)}$ and the low-accuracy amplitudes $\{B_{{h_0},n}\}_{n=1}^{N_{h_0}(\mathbf{r}_0)}$ as in the first part of Subsection 3.1.

 {\bf Step 3}. Let $\mathcal{T}_{\widetilde{h}}$ be a uniformly refining triangular mesh of $\mathcal{T}_{h_0}$ with the mesh size $\tilde{h}$ satisfying $\widetilde{\omega}\widetilde{h} \approx 3N_{h_0}(\mathbf{r}_0)c(\mathbf{r}_0)$. We apply the GOPWDG method to the discretization of (\ref{lowfhomoIBC}) on $\mathcal{T}_{\widetilde{h}}$, and use $u_{\widetilde{h}}(\widetilde{\omega}, \mathbf{r})$ to denote the resulting approximate solution.

 {\bf Step 4}. Consider every element of $\mathcal{T}_{\widetilde{h}}$. Let $N_{\widetilde{h}}(\mathbf{r}_0)$, $\{\theta_{{\widetilde{h}},n}^{\ast}\}_{n=1}^{N_{\widetilde{h}}(\mathbf{r}_0)}$ and $\{B_{{\widetilde{h}},n}^{\ast}\}_{n=1}^{N_{\widetilde{h}}(\mathbf{r}_0)}$  be the natural interpolation of $N_{h_0}(\mathbf{r}_0)$, $\{\theta_{{h_0},n}\}_{n=1}^{N_{h_0}(\mathbf{r}_0)}$ and $\{B_{{h_0},n}\}_{n=1}^{N_{h_0}(\mathbf{r}_0)}$, respectively. We replace the function $u(\widetilde{\omega}, \mathbf{r})$ in (\ref{secapp}) by $u_{\widetilde{h}}(\widetilde{\omega}, \mathbf{r})$
 to get two new impedance quantities, and further sample the quantities to compute the numbers $\{\widetilde{U}_{\widetilde{h}, \ell}\}_{\ell\leq3N_{\widetilde{h}}(\mathbf{r}_0)}$. We use these numbers to define the minimization
  problem (\ref{nonlinearmin}), and solve it by the DLS method
 % minimization algorithm
  to obtain high-accuracy ray angles $\{\theta_{{\widetilde{h}},n}^{post}\}_{n=1}^{N_{\widetilde{h}}(\mathbf{r}_0)}$, where the initial guesses are chosen as $\theta_{{\widetilde{h}},n}^{\ast}$ and $B_{{\widetilde{h}},n}^{\ast}$.

{\bf Step 5}. Let $\mathcal{T}_{h}$ be a uniformly refining triangulation of $\mathcal{T}_{\widetilde{h}}$ with the mesh size $h$ satisfying $\omega h\approx 1$. Let $N_h(\mathbf{r}_0)$ and $\{\theta_{h,n}\}_{n=1}^{N_h(\mathbf{r}_0)}$ denote the natural interpolations of $N_{\widetilde{h}}(\mathbf{r}_0)$ and $\{\theta_{{\widetilde{h}},n}^{post}\}_{n=1}^{N_{\widetilde{h}}(\mathbf{r}_0)}$ on $\mathcal{T}_{h}$,
respectively.

Suppose we have obtained exact ray numbers ($N_{\widetilde{h}}(\mathbf{r}_0) = N(\mathbf{r}_0)$) on each element. We give the approximation property of the resulting ray angles in the following theorem.
\begin{theorem}\label{multiraylemma}
The resulting ray angles of the above process have the following approximation error
\begin{equation}\label{multirayerr}
	|\theta_{h,n}^{post} - \theta_n| = \mathcal{O}(\omega^{-1}), \quad n = 1,\cdots, N(\mathbf{r}_0).
\end{equation}
\end{theorem}
\begin{proof}
Notice that the GOPWDG  approximate solutions of the low-frequency problem possess sufficient high accuracy. Applying (\ref{theomultiray}) to every $\theta_{{\widetilde{h}},n}^{post}$ and using the definitions
%interpolation error $\mathcal{O}(h)=\mathcal{O}(\omega^{-1})$
given in the above {\bf Step 5}, we obtain the approximate error of the resulting ray angles.
\end{proof}

\subsection{Approximation property of the adaptive plane wave space}

Define $\mathbf{d}_{h,n} = (\cos\theta^{post}_{h, n}, \sin\theta^{post}_{h, n}), n = 1, \cdots, N(\mathbf{r}_0)$. On the element $K_0$ with the barycentric points denoted by $\mathbf{r}_0$, our adaptive GOPW basis functions are chosen as
\begin{equation}\label{highmultibasis}
	\varphi_{n, j}(\omega, \mathbf{r}) = p^{\tau_{h,n}}_{j}({\bf r})e^{i\omega\tau_{h,n}(\mathbf{r})}, \quad n = 1, \cdots, N(\mathbf{r}_0), j = 1, 2.
\end{equation}
with $\tau_{h,n}({\bf r})$ and $p^{\tau_{h,n}}_{j}({\bf r})$ ($j=1,2$) be the polynomials determined as in Subsection 2.2 by replacing the discrete plane wave direction
${\bf d}_n$
with $\mathbf{d}_{h,n}$. Our adaptive GOPW space adapted to solve the high-frequency problem is defined as
$$ V_{ray}(K_0)=span\big\{\varphi_{n, j}(\omega, \mathbf{r}):~n = 1, \cdots, N(\mathbf{r}_0);~j=1,2\big\}. $$
In order to investigate approximate properties of this space, we define
$$ {\mathcal H}^2(K_0)=\big\{v\in H^2(K_0):~~v~~\mbox{satisfies}~~(\Delta+\kappa I)=0~~\mbox{on}~~K_0\big\}. $$

\begin{theorem}\label{approx} Let $u\in {\mathcal H}^2(K_0)$ be the analytic solution of (\ref{homohelm}).
Assume that $\phi_n$ and $A_{n,s}$ in the optics ansatz (\ref{geoanaN})-(\ref{geoanan}) of $u$ satisfy $\phi_n\in C^2(K_0)$ and $A_{n,s}\in C^1(K_0)$.
%Suppose that the mesh size $h$ satisfies $h \omega\leq C$.
Then there exists an interpolation operator $\pi_h: {\mathcal H}^2(K_0)\rightarrow V_{ray}(K_0)$ such that
\begin{equation}\label{local-approx1}
\|u-\pi_h u\|_{L^2(K_0)}\leq C(\omega^{-1}+h+\omega h^2),
\end{equation}
%and
%\begin{equation}\label{local-approx2}
%\|u-u_h\|_{C^1(K_0)}\leq C.
%\end{equation}
where the constant $C$ depends only on the upper bound of $\|\phi_n\|_{H^2(K_0)}$ and $\|A_{n,s}\|_{H^1(K_0)}$. In particular, when choosing $h$ as $h\sim \omega^{-1}$, we have
$$ \|u-\pi_h u\|_{L^2(K_0)}\leq C(\omega^{-1})=C h. $$
\end{theorem}
\begin{proof} Since the basis functions in $V_{ray}(K_0)$ are independent for different wave directions, we can only consider the case that $u$ has only one wave direction on $K_0$,
i.e., $n=1$. For ease of notation, we simply write (see Subsection 2.1)
\begin{equation}\label{3.optical-ansatz1}
u({\bf r})=A({\bf r})e^{i\omega\phi({\bf r})}
%+\varepsilon(\omega,{\bf r})
%O(\omega^{-1})
,\quad A({\bf r})=\sum\limits_{s=0}^{\infty}(i\omega)^{-s}A_s({\bf r})
\quad {\bf r}\in K_0,
\end{equation}
%where $\varepsilon(\omega,{\bf r})$ satisfies
%$$ |\varepsilon(\omega,{\bf r})|=O(\omega^{-1})\quad \mbox{and} \quad |\nabla\varepsilon(\omega,{\bf r})|=O(1). $$

Let $\theta$ be the exact direction angle defined by $\phi$ and set ${\bf d}_{\theta}=(cos\theta~ \sin\theta)^t$, which satisfies $\nabla\phi({\bf r}_0)=c^{-1}({\bf r}_0){\bf d}_{\theta}$ since
$|\nabla \phi|^2=\xi=c^{-2}$. Then, by the Taylor formula, the phase function $\phi$ can be written as
$$ \phi({\bf r})=\phi({\bf r}_0)+P_1^{\theta}({\bf r})+O(|{\bf r}-{\bf r}_0|^2), \quad {\bf r}\in K_0, $$
where $P_1^{\theta}({\bf r})$ is a linear polynomial of $x-x_0$ and $y-y_0$ and it can be written as
$$ P_1^{\theta}({\bf r})=c^{-1}({\bf r}_0){\bf d}_{\theta}\cdot({\bf r}-{\bf r}_0)=c^{-1}({\bf r}_0)(cos\theta(x-x_0)+\sin\theta(y-y_0)). $$
%Moreover, it follows by the mean value formula that
%$$ A({\bf r})=A({\bf r}_0)+O(|{\bf r}-{\bf r}_0|), \quad {\bf r}\in K_0. $$

Let $\theta_h$ be the approximation direction angle defined in the last subsection, which satisfies
\begin{equation}\label{3.approx1}
|\theta_h-\theta|=O(\omega^{-1}).
\end{equation}
By the definitions of $\tau_{h}$ (see Subsection 2.2), we have $\tau_{h}({\bf r})=P_1^{\theta_h}({\bf r})$ (${\bf r}\in K_0)$
with
$$ P_1^{\theta_h}({\bf r})=c^{-1}({\bf r}_0){\bf d}_{\theta_h}\cdot({\bf r}-{\bf r}_0)=c^{-1}({\bf r}_0)(cos\theta_h(x-x_0)+\sin\theta_h(y-y_0)). $$
Moreover, from the definition of the polynomials $\{p^{\tau_{h}}_{j}\}_{j=1}^2$ (see Subsection 2.2), we have
$$ p^{\tau_{h}}_1({\bf r})=1+{\bf d}^{\bot}_{\theta_h}\cdot({\bf r}-{\bf r}_0)\quad\mbox{and}\quad p^{\tau_{h}}_2({\bf r})=1-{\bf d}^{\bot}_{\theta_h}\cdot({\bf r}-{\bf r}_0). $$
Let $\gamma_{K_0}(A_0)$ denote the integration average of $A_0$ on $K_0$, and define $x_1=x_2={1\over 2}\gamma_{K_0}(A_0)$.
Then
$$ x_1p^{\tau_{h}}_{1}({\bf r})+x_2 p^{\tau_{h}}_{2}({\bf r})=\gamma_{K_0}(A_0). $$
Thus
$$ |A_0({\bf r})-x_1p^{\tau_{h,n}}_{1}({\bf r})+x_2 p^{\tau_{h,n}}_{2}({\bf r})|=|A_0({\bf r})-\gamma_{K_0}(A_0)|, \quad {\bf r}\in K_0,$$
which implies that
\begin{equation}\label{3.approx3}
|A_0e^{i\omega\phi({\bf r}_0)}-e^{i\omega\phi({\bf r}_0)}\sum\limits_{j=1}^2x_jp^{\tau_{h,n}}_{j}|=|A_0({\bf r})-\gamma_{K_0}(A_0)|,\quad \mbox{on}~K_0.
\end{equation}
Define
$$ \pi_h u({\bf r})=e^{i\omega\phi({\bf r}_0)}\sum\limits_{j=1}^2x_jp^{\tau_{h}}_{j}({\bf r})e^{i\omega\tau_{h,n}(\mathbf{r})}. $$
Notice that
$$ u({\bf r})=A_0 e^{i\omega\phi({\bf r}_0)}e^{i\omega P_1^{\theta}+O(\omega h^2)}+O(\omega^{-1}). $$
Then
\begin{eqnarray*}
u({\bf r})-\pi_h u({\bf r})&=&\big(A_0 e^{i\omega\phi({\bf r}_0)}-e^{i\omega\phi({\bf r}_0)}\sum\limits_{j=1}^2x_jp^{\tau_{h}}_{j}\big)e^{i\omega P_1^{\theta}+O(\omega h^2)}\cr
&+&e^{i\omega\phi({\bf r}_0)}\sum\limits_{j=1}^2x_jp^{\tau_{h}}_{j}\big(e^{i\omega P_1^{\theta}+O(\omega h^2)}-e^{i\omega\tau_{h}(\mathbf{r})}\big)+O(\omega^{-1}).
\end{eqnarray*}
Therefore, by (\ref{3.approx3}) we have
\begin{equation}\label{3.approx4}
|u({\bf r})-\pi_h u({\bf r})|\leq C\big(|A_0({\bf r})-\gamma_{K_0}(A_0)|
+|\gamma_{K_0}(A_0)||\tau_{h}(\mathbf{r})-P_1^{\theta}(\mathbf{r})|\big)+O(\omega^{-1}).
\end{equation}
It follows by (\ref{3.approx1}) that
$$ |\tau_{h}(\mathbf{r})-P_1^{\theta}(\mathbf{r})|\leq C\|\phi\|_{H^2(K_0)}(\omega^{-1}h)+O(h^2),\quad \mathbf{r}\in K_0.$$
Substituting this into (\ref{3.approx4}), we obtain
$$ \|u-\pi_h u\|_{L^2(K_0)}\leq C\big(\|A_0\|_{H^1(K_0)}+\|\phi\|_{H^2(K_0)}\big)(h+\omega h^2)+O(\omega^{-1}), $$
which gives the desired result.
\end{proof}
\begin{remark} Under the assumption that $h\sim\omega^{-1}$ and the upper bound of $\|\phi_n\|_{H^2(K_0)}$ and $\|A_{n,s}\|_{H^1(K_0)}$ is independent of $\omega$, the
adaptive plane wave space $V_{ray}(K_0)$ possesses $1-$order convergence with respect to $h$ or $\omega^{-1}$ and has much better convergence than existing
discrete spaces for Helmholtz equations with large wave numbers.
\end{remark}

\subsection{A discretization method of (\ref{nonhomogeneousHelm})}

The adaptive plane wave space $V_{ray}(K_0)$ was designed for homogeneous Helmholtz equations on the element $K_0$. In order to use this space to the discretization of the nonhomogeneous Helmholtz equation (\ref{nonhomogeneousHelm}), we need to adopt the plane wave method combined
with local spectral elements (PW-LSFE) first presented in \cite{Huqy2018}, which was extended to the case with variable wave numbers in \cite{HuW2021}. To shorten the length of this paper,
here we only describe the basic idea of the method (more details can be found in \cite{HuW2021}, where only a different plane wave space was used).

Assume that $f$ is defined in a slightly large domain containing $\Omega$ as its subdomain and the domain $\Omega$ is strictly star-shaped.
As usual, let $\Omega$ be decomposed into the union of some elements $\{\Omega_k\}_{k=1}^N$, which constitute quasi-uniform and and shape-regular triangulations $\mathcal{T}_h$ with mesh sizes $h$.
For each element $\Omega_{k}$, we choose a disc domain $\Omega_{k}^{*}$ that has almost the same size of $\Omega_{k}$ and contains $\Omega_{k}$ as its subdomain.

Let $u_{k}^{(1)} \in H^{1}(\Omega_{k}^{*})$ be the solution of the restriction of the nonhomogeneous Helmholtz equation (\ref{nonhomogeneousHelm}) on $\Omega_{k}^{*}$, with homogeneous Robin boundary condition on $\partial\Omega_{k}^{*}$.
Define $u^{(1)} \in L^{2}(\Omega)$ as $u^{(1)}|_{\Omega_{k}}=u_{k}^{(1)}|_{\Omega_{k}}$ for each $\Omega_{k}$.
Set $u^{(2)}=u-u^{(1)}$, then $u^{(2)}$ satisfies a homogeneous Helmholtz equation on every element $\Omega_k$. Notice that $u^{(2)}$ satisfies two transmission conditions depending on $u^{(1)}$ on the common edge
of two neighboring elements.

Let $m$ be a positive integer, and let $S_m(\Omega_{k}^{*})$ denote the set of polynomials defined on $\Omega_{k}^{*}$, whose orders are less or equal to $m$.
We use $u_{k, h}^{(1)}\in S_m(\Omega_{k}^{*})$ to denote the standard spectral element solution of the nonhomogeneous Helmholtz equation satisfied by $u_{k}^{(1)}$, and
define $u_{h}^{(1)} \in \prod_{k=1}^{N} S_{m}(\Omega_{k})$ by $u_{h}^{(1)}|_{\Omega_{k}}=u_{k, h}^{(1)} |_{\Omega_{k}}$.

We use $V_{ray}(\mathcal{T}_{h})$ to denote the adaptive plane wave space spanned by the local basis functions $\varphi_{n,j}(\omega,{\bf r})$. Namely, the space $V_{ray}(\mathcal{T}_{h})$ is defined as
$$ V_{ray}({\mathcal T}_h)=\{v_h\in L^2(\Omega):~v_h|_K\in V_{ray}(K),~\forall K\in {\mathcal T}_h\}. $$
Let $u_{h}^{(2)} \in V_{ray}(\mathcal{T}_{h})$ be the approximation of $u^{(2)}$, where $u_{h}^{(2)}$ are defined by the discontinuous Galerkin method with
$V_{ray}({\mathcal T}_h)$ for the local homogeneous Helmholtz equation satisfied by $u^{(2)}$.

The final approximate solution $u_h\in L^2(\Omega)$ is defined by $u_h|_{\Omega_k} = u_h^{(1)}|_{\Omega_k} + u_h^{(2)}|_{\Omega_k}$. For convenience, we call this method as ray-GOPWDG-LSFE method.
In most situations, the number $N({\bf r}_0)$ of the wave directions is small, so the space $V_{ray}({\mathcal T}_h)$ has much smaller degrees of freedom than
the standard plane wave space $V_p({\mathcal T}_h)$, where $p$ must increase when $\omega$ increases.
\section{Numerical experience} \label{numerical}
In this section we apply the ray-GOPWDG-LSFE method to solve the nonhomogeneous Helmholtz equations with variable wave numbers
\begin{equation}\label{numnonhomo}
\left\{\begin{aligned}
	&-\Delta u - \omega^2\xi^2 u = f \quad \text{in $\Omega$}, \\
	&(\frac{\partial}{\partial \mathbf{n}} + i\omega\xi) u = g \quad \text{on $\partial\Omega$},
\end{aligned}\right.
\end{equation}
 and we report some numerical results to confirm the effectiveness of the proposed methods.

Let $\Omega$ be divided into small rectangles. Each rectangle has the same mesh size $h$, where $h$ is the length of the longest edge of the elements. The resulting uniform triangulation is denoted by $\mathcal{T}_h$.
We fix the number of elements per wavelength for the experiments made in this subsection, i.e., we choose $h=\omega^{-1}$. Moreover,
we only consider two examples where the ray number is a constant among the whole computed area, then the number of basis functions on every element is fixed.
% The resulting discrete system is solved by the simplest overlapping domain decomposition method without coarse solver.

%In order to illustrate the efficiency of the ray-GOPWDG-LSFE method, we state the number of elements (and degrees of freedom) per wavelength for the experiments made in this subsection. Let $n_e$ denote the number of basis %functions per element.

We introduce the relative $L^2$ error
\begin{equation*}
	\text{Err.} = \frac{\|u_{ex} - u_h\|_{L^2(\Omega)}}{\|u_{ex}\|_{L^2(\Omega)}},
\end{equation*}
where $u_{ex}$ is the analytic solution and $u_h$ is the numerical solution. Define $\delta$ by
\begin{equation*}
	\delta = \frac{\log(\text{Err}_2/\text{Err}_1)}{\log(\omega_2/\omega_1)},
\end{equation*}
which can measure the ``pollution effect" of a numerical method.
%We will compare the ray-GOPWDG-LSFE method and ray-FEM method

%In case of this two methods, the number of the discretized plane wave directions may be different in each element and it locally depends on the number of rays passing by. In this section, we give simplified numerical examples %where the ray number is a constant among the whole computed area. The resulting discrete system are solved by overlapping domain decomposition method.

%Let $n_e$ denote the number of basis functions per element.
Let $\theta(\hat{\mathbf{d}}_{ex})$, $\theta(\hat{\mathbf{d}}^F_h)$ and $\theta(\hat{\mathbf{d}}^G_{\widetilde{\omega}})$ denote the exact ray angles, the numerical ray angles computed in ray-FEM method and
 the numerical ray angles computed by the post-processing method in ray-GOPWDG-LSFE method, respectively. Denote by $u^F_h$ the numerical solution solved by the ray-FEM method. Let $u^G_{\hat{\mathbf{d}}_{\widetilde{\omega}}}$ and $u_{\hat{\mathbf{d}}_{ex}}$ denote the numerical solution solved by the ray-GOPWDG-LSFE method used by the direction $\hat{\mathbf{d}}^G_{\widetilde{\omega}}$ and $\hat{\mathbf{d}}_{ex}$, respectively. Let ``DOFs" denotes the computing complexity of the considered numerical methods.

We compare the performances of the ray-GOPWDG-LSFE method and the ray-FEM method for the high-frequency numbers $\omega = 400, 625, 900$.

\subsection{Example 1 (Single wave in a heterogeneous medium)}\label{example1}
We consider an example in a heterogeneous medium in the domain $\Omega = [0,1]\times[0,1]$ (see \cite{Fomel2009}): Define $\xi(\mathbf{r}) = \frac{1}{c(\mathbf{r})^2}$ with the velocity field $c(\mathbf{r})$  as a smooth converging lens with a Gaussian profile at the center $(x_0, y_0) = (1/2, 1/2)$
$$c(x, y)=\frac{4}{3}\left(1-\frac{1}{8} \exp \left(-32\left(\left(x-r_{1}\right)^{2}+\left(y-r_{2}\right)^{2}\right)\right)\right). $$
The analytic solution of the problem is given by
$$u_{ex} (x, y) = c(x, y)e^{i\omega xy}$$
Then the source term is $f_{ex} = -\Delta u - \omega^2\xi(\mathbf{r}) u$ and the boundary function is chosen as $g_{ex} = (\frac{\partial}{\partial\mathbf{n}} + i\omega)u_{ex}$.

Following the process Step 1-Step 2 in Subsection \ref{multinum}, we first solve the corresponding low frequency-problem $\widetilde{\omega} = 20, 25, 30$ on a coarsen mesh $\widetilde{h}$, respectively.
\begin{figure}[!ht]
\centering
\includegraphics[width=9cm,height=6cm]{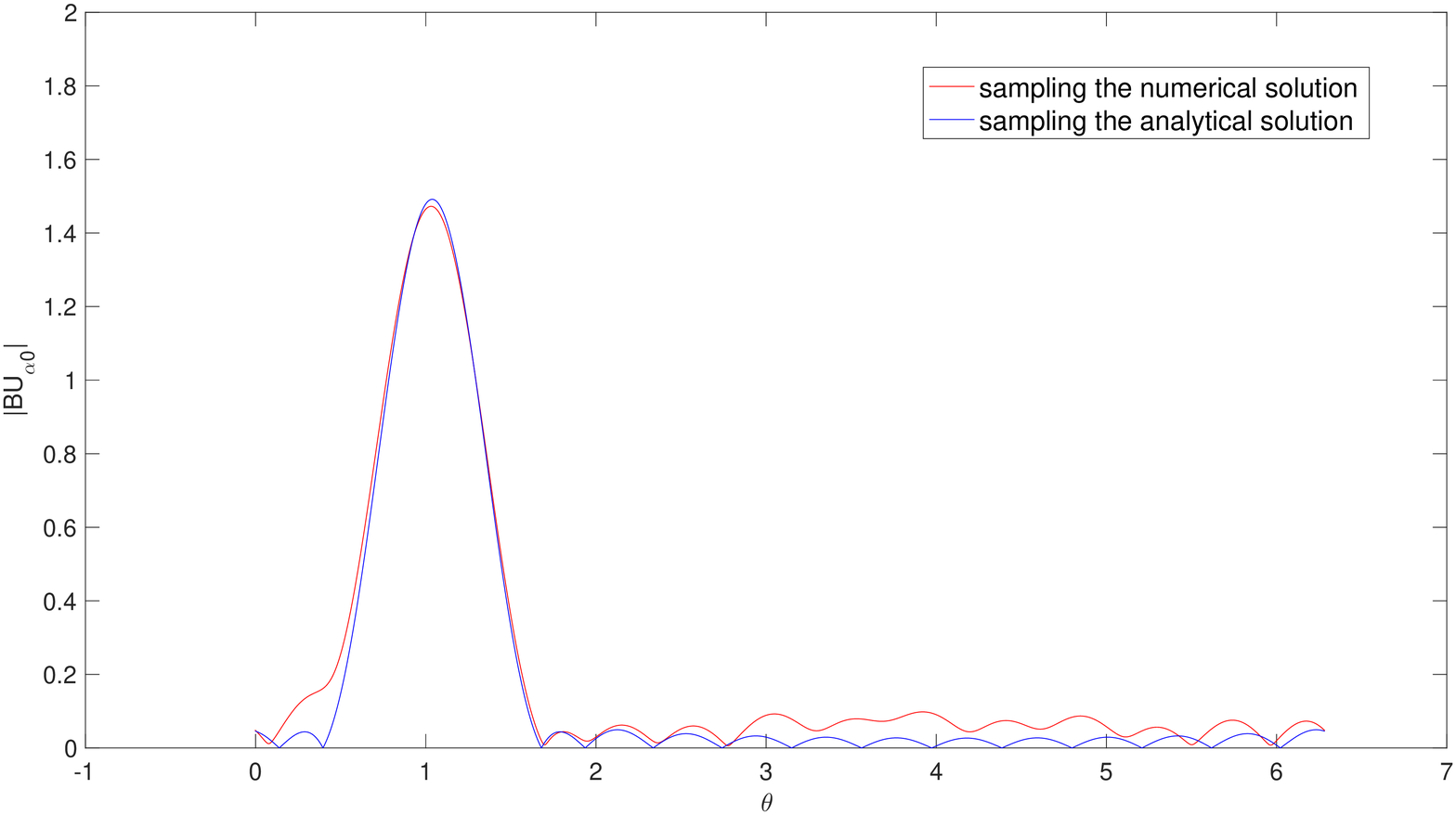}
\caption{NMLA sampling at the location (19/20, 3/20)}
\label{sampling2fig}
\end{figure}
The NMLA sampling results shows that there is only one ray locally at each element (see Figure \ref{sampling2fig}).
We then use the post-processing method locally to calculate the ray directions in each element.

The maximum iteration steps cost by the DLS method are reported in Table \ref{varyDLSiter1}.

\begin{table}[H]
\caption{Maximum iteration steps of the DLS method in the post-processing method}
\label{varyDLSiter1}
\begin{tabular}{cccc}
\hline
$\omega$ & 400 & 625 & 900 \\ \hline
iter(s)  & 33  & 34   & 31    \\ \hline
\end{tabular}
\end{table}

In Table  \ref{1varyraycom2}, we list $L^{\infty}$ errors $\|\theta(\hat{\mathbf{d}}_{\widetilde{\omega}})-\theta(\hat{\mathbf{d}}_{ex})\|_{\infty}$ of the ray angles.

\begin{table}[!h]
\centering
\caption{Complexity and ray angle approximation errors of the ray-GOPWDG-LSFE method and ray-FEM method}
\label{1varyraycom2}
\begin{tabular}{lllllll}
\hline
$\omega$ & Comp. & $\|\theta(\hat{\mathbf{d}}^G_{\widetilde{\omega}})-\theta(\hat{\mathbf{d}}_{ex})\|_{\infty}$ & Order  & Comp. & $\|\theta(\hat{\mathbf{d}}^F_h)-\theta(\hat{\mathbf{d}}_{ex})\|_{\infty}$ & Order\\ \hline
400      &   9.7e+5   	&   4.321e-3       &   $-$    & 1.3e+8&  5.242e-3 & $-$\\ \hline
625     &  2.4e+6    &   2.635e-3       &    1.11      &  4.4e+8&  3.435e-3&   0.95    \\ \hline
900    &  4.9e+6   &    1.839e-3      &     0.98       &  1.2e+9&  2.384e-3 &   1.00\\ \hline
\end{tabular}
\end{table}

The relative $L^2$ errors of the approximated solutions and the total DOFs needed in the ray-GOPWDG-LSFE method are reported in the Table  \ref{1varyrayu2}.

\begin{table}[H]
\centering
\caption{DOFs. and Approximation errors of the ray-GOPWDG-LSFE method and ray-FEM method}
\label{1varyrayu2}
\begin{tabular}{ccccccc}
\hline
$\omega$ & DOFs & $\|u^G_{\hat{\mathbf{d}}_{\widetilde{\omega}}}-u\|_{0, \Omega}$ & Order & DOFs & $\|u_h^F-u\|_{0, \Omega}$ & Order \\ \hline
400        & 6.6e+4 &  5.817e-4    &   $-$          & 1.3e+7 & 7.113e-4      & $-$	      \\ \hline
625        & 1.6e+5 &  3.592e-4    &   1.08        & 4.1e+7 & 4.615e-4      & 0.96        \\ \hline
900        & 3.4e+5 &  1.671e-4    &   1.04        & 6.7e+7  &  3.214e-4     &  0.99       \\ \hline
\end{tabular}
\end{table}

It shows that for the single wave solution, the approximation error of the numerical solutions of the both methods can have an optimal convergence of $\mathcal{O}(\omega^{-1})$. However, the total DOFs of the ray-GOPWDG-LSFE method is $\mathcal{O}(\omega^2)$, which is also optimal in solving the two-dimensional Helmholtz equation and much less than the total DOFs needed by the ray-FEM method.

\subsection{Example 2. (Constant gradient of slowness squared $\xi$)}\label{example2}
We provide an example in a heterogeneous medium with wave speed of constant gradient (see \cite{Fomel2009}): $\xi(\mathbf{r}) = c_0^2+2\mathbf{G}_0\cdot(\mathbf{r}-\mathbf{r}_0)$ with parameters $c_0 = 1$, $\mathbf{G}_0 = (0.1, -0.2)$ and $\mathbf{r}_0 = (-0.1, -0.1)$. Referring to \cite{Fomel2009}, there are two rays crossing in the domain $\Omega = [0,1]\times[0,1]$. The two phase functions are known analytically and they are given by
\begin{equation}
	\phi_j = \bar{c}\sigma_j - \frac{|\mathbf{G}_0|^2}{6}\sigma_j^3, \quad j = 1,2,
\end{equation}
where
\begin{equation}
	\sigma_j = \frac{\sqrt{2(\bar{c} + (-1)^j\sqrt{\bar{c}^2-|\mathbf{G}_0|^2|\mathbf{r}-\mathbf{r}_0|^2})}}{|\mathbf{G}_0|^2}, \quad j = 1,2,
\end{equation}
with
\begin{equation}
	\bar{c} = c_0 + \mathbf{G}_0\cdot(\mathbf{r}-\mathbf{r}_0).
\end{equation}
Then the analytic solution is given by
 \begin{equation}
	u_{ex} = exp(i\omega\phi_1)/(xy+1i) + exp(i\omega\phi_2)/((x^2+y^2+1i)).
\end{equation}
The source term is $f_{ex} = -\Delta u - \omega^2\xi(\mathbf{r}) u$ and the boundary function is chosen as $g_{ex} = (\frac{\partial}{\partial\mathbf{n}} + i\omega)u_{ex}$.

Following the process Step 1-Step 2 in Subsection \ref{multinum}, we first solve the corresponding low frequency-problem $\widetilde{\omega} = 20, 25, 30$ on a coarsen mesh $\widetilde{h}$, respectively.
\begin{figure}[!ht]
\centering
\includegraphics[width=9cm,height=6cm]{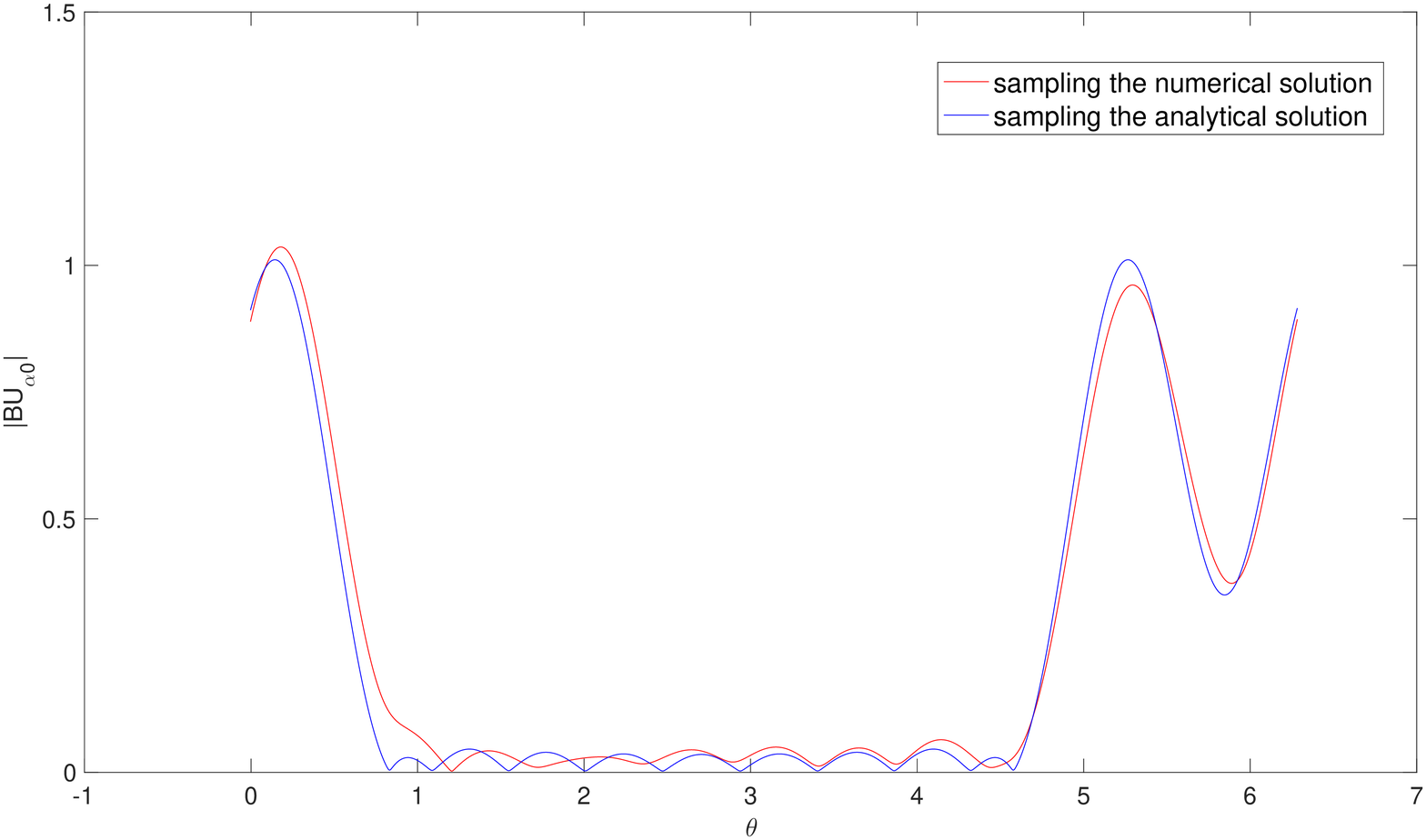}
\caption{NMLA sampling at the location (19/20, 3/20)}
\label{sampling3fig}
\end{figure}
It shows that we can learn nearly-accurate ray directions based on a numerical solution generated by the GOPWDG-LSFE method on the mesh $\widetilde{h}  = \mathcal{O}(\widetilde{\omega}^{-1})$ (see Figure \ref{sampling3fig}).

We then use the post-processing method locally to calculate the ray directions in each element. The maximum iteration steps needed in the DLS method are listed in Table \ref{varyDLSiter}.

\begin{table}[!h]
\caption{Maximum iteration steps of the DLS method in the post-processing method}
\label{varyDLSiter}
\begin{tabular}{llll}
\hline
$\omega$ & 400 & 625 & 900 \\ \hline
iter(s)  & 36  & 35   & 33    \\ \hline
\end{tabular}
\end{table}

The $L^{\infty}$ errors of the ray directions are reported in Table \ref{varyraycom3}.
\begin{table}[!h]
\centering
\caption{Complexity and ray angle approximation errors of the ray-GOPWDG-LSFE method and ray-FEM method}
\label{varyraycom3}
\begin{tabular}{lllllll}
\hline
$\omega$ & Comp. & $\|\theta(\hat{\mathbf{d}}^G_{\widetilde{\omega}})-\theta(\hat{\mathbf{d}}_{ex})\|_{\infty}$ & Order  & Comp. & $\|\theta(\hat{\mathbf{d}}^F_h)-\theta(\hat{\mathbf{d}}_{ex})\|_{\infty}$ & Order\\ \hline
400      &   1.3e+6   	&   8.433e-3       &   $-$    & 1.9e+8&  1.205e-1 & $-$\\ \hline
625     &  3.1e+6    &   5.133e-3       &    1.11      &  6.3e+8&  9.238e-2&   0.60    \\ \hline
900    &  6.5e+6   &    3.482e-3      &     1.06       &  1.7e+9&  7.662e-2 &   0.51\\ \hline
\end{tabular}
\end{table}

The relative $L^2$ errors of the approximated solutions and the total DOFs needed of the ray-GOPWDG-LSFE method are shown in the Table  \ref{varyrayu3}.
\begin{table}[H]
\centering
\caption{Dofs. and approximation errors of the ray-GOPWDG-LSFE method and ray-FEM method}
\label{varyrayu3}
\begin{tabular}{ccccccc}
\hline
$\omega$ & DOFs & $\|u^G_{\hat{\mathbf{d}}_{\widetilde{\omega}}}-u\|_{0, \Omega}$ & Order & DOFs & $\|u_h^F-u\|_{0, \Omega}$ & Order \\ \hline %& $\|u_{\hat{\mathbf{d}}_{ex}}-u\|_{0, \Omega}$ & Order \\ \hline
400        & 1.3e+6 &  3.934e-4    &   $-$          & 1.9e+7 &3.391e-3      & $-$	      \\ \hline%&       2.609e-4  & $-$
625        & 3.1e+6 &   2.327e-4   &   1.176      &  6.3e+7 &2.516e-3      & 0.67        \\ \hline%&       1.671e-4 & 1.00
900        & 6.5e+6 &   1.562e-4   &   1.094      &  1.7e+7 &2.049e-3      & 0.56        \\ \hline%&        1.159e-4 & 1.00
\end{tabular}
\end{table}

It shows that our ray-GOPWDG-LSFE method are more efficient than the ray-FEM method for high-frequency Helmholtz problem and the relative $L^2$ errors have the optimal convergence of $\mathcal{O}(\omega^{-1})$. The total DOFs of the ray-GOPWDG-LSFE method is $\mathcal{O}(\omega^2)$, which is also optimal in solving the two-dimensional Helmholtz equation.
%However, the ray-FEM method possesses much more total DOFs and only obtains an approximated error of $\mathcal{O}(\omega^{-1/2})$.

\section{Conclusion}
In this paper we have introduced an adaptive ray-based GOPW method for the high-frequency Helmholtz equation in smooth media. We have developed different ray-learning method for the single wave as well as the multiple wave. We have derived an interpolation error of the ray-GOPW spaces. The numerical results shows that the ray GOPW method can achieve asymptotic convergence rate of $\mathcal{O}(\omega^{-1})$ for multiple waves. The computing complexity can be also optimal as $\mathcal{O}(\omega^2)$ in two dimensions.

\bibliographystyle{siamplain}

\end{document}